\documentclass[twoside, 11pt]{article}
\usepackage{geometry}
\usepackage{times}
\usepackage[utf8]{inputenc}
\usepackage[vietnamese,english]{babel}
\usepackage[displaymath, tightpage]{preview}
\usepackage[all, cmtip]{xy}

\usepackage{amsmath, bm}
\usepackage{amsxtra}
\usepackage{amscd}
\usepackage{amsthm}
\usepackage{amsfonts}
\usepackage{amssymb}
\usepackage{bbm}
\usepackage{mathtools}
\usepackage[bitstream-charter]{mathdesign}
\usepackage[T1]{fontenc}

\usepackage{fancyhdr}
\usepackage{tikz}
\usepackage{tikz-cd}
\usetikzlibrary{matrix,calc,arrows}
\usepackage[new]{old-arrows}

\usepackage[hidelinks]{hyperref}

\usepackage[scaled=0.8]{helvet} 

\newcommand\dqd{\foreignlanguage{vietnamese}{Đinh Quý Dương}}

\numberwithin{equation}{section}
\usepackage{sectsty}
\sectionfont{\fontsize{13}{15}\selectfont}

\newtheoremstyle{theorem}
{3mm}
{1mm}
{\normalfont\itshape}
{}
{\normalfont\scshape}
{.}
{.5em}
{\thmname{#1}\thmnumber{ #2}\thmnote{ (#3)}}
\theoremstyle{theorem}
\newtheorem{theorem}{Theorem}
\newtheorem{lemma}[theorem]{Lemma}
\newtheorem{proposition}[theorem]{Proposition}
\newtheorem{corollary}[theorem]{Corollary}

\numberwithin{theorem}{section}

\newtheoremstyle{definition}
{3mm}
{1mm}
{\normalfont\normalfont}
{}
{\normalfont\scshape}
{.}
{.5em}
{\thmname{#1}\thmnumber{ #2}\thmnote{ (#3)}}
\theoremstyle{definition}
\newtheorem{remark}[theorem]{Remark}
\newtheorem{definition}[theorem]{Definition}
\newtheorem{example}[theorem]{Example}

\newcommand{\cD}{{\mathcal{D}}}
\newcommand{\cE}{{\mathcal{E}}}
\newcommand{\cL}{{\mathcal{L}}}

\newcommand{\divisor}{{\mathrm{div}}}

\newcommand{\Res}{{\mathrm{Res}}}
\newcommand{\bC}{{\mathbb{C}}}
\newcommand{\bP}{{\mathbb{P}}}
\newcommand{\SOV}{{\mathrm{SoV}}}
\newcommand{\SOVdR}{{\mathrm{SoV_{\lambda}}}}
\newcommand{\connref}{^{\mathrm{ref}}}

\newcommand{\mM}{{\mathcal{M}}}
\newcommand{\M}{\mM_{d}}
\newcommand{\MdRd}{{\mM^d_\lambda}}
\newcommand{\tMdRd}{{\widetilde{\mM^d_{\lambda}}}}

\newcommand{\Nstable}{\mathcal{N}}
\newcommand{\N}{\mathcal{N}_{d}}

\newcommand{\Pic}{\mathrm{Pic}}
\newcommand{\Mop}{{\mM^m_{\mathrm{op}, \lambda}}}
\newcommand{\MHodge}{{\mathcal{M}_{Hod}}}

\newcommand{\bp}{{\mathbf{p} }}

\newcommand{\bq}{{ \mathbf{q} }}

\newcommand{\bqcheck}{\mathbf{\check{q}}}

\newcommand{\bmx}{{ \bm{x} }}
\newcommand{\bmr}{{ \bm{r} }}
\newcommand{\bmk}{{ \bm{k} }}
\newcommand{\blambda}{{ \bm{\lambda} }}
\newcommand{\bkappa}{{ \bm{\kappa} }}
\newcommand{\bmz}{\bm{z}}
\newcommand{\bmzcheck}{\bm{\check{z}}}
\newcommand{\bu}{{ \mathbf{u} }}
\newcommand{\bmu}{{ \bm{u} }}

\newcommand{\tbu}{ \mathbf{\tilde{u}} }

\title{From $\lambda$-connections \\ to $PSL_2(\mathbb{C})$-opers with apparent singularities}
\author{\dqd}
\date{}
\begin{document}
\maketitle

\begin{abstract}
On a Riemann surface of genus $> 1$,
we discuss how to construct opers with apparent singularities
from $SL_2(\mathbb{C})$ $\lambda$-connections $(E, \nabla_\lambda)$ 
and sub-line bundles $L$ of $E$.
This construction defines a rational map from a space
which captures important data of triples $(E, L, \nabla_\lambda)$ to a space
which parametrises the positions and residue parameters of the induced 
apparent singularities.
We show that this is a Poisson map with respect to natural Poisson structures.
The relations to wobbly bundles and Lagrangians in the moduli spaces of Higgs bundles 
and $\lambda$-connections are discussed.
\end{abstract}

\section{Introduction}
Let $X$ be a compact connected Riemann surface of genus $g \geq 2$, 
and $G$ a simple complex Lie group.
The de Rham moduli space $\mM_{dR}$ of irreducible holomorphic $G$-connections $(E, \nabla)$
on $X$ carries a hyperkähler structure,
and in particular has an associated twistor space $\mM_{dR} \times \bP^1$.
The fiber $\mM_\lambda$ over a fixed $\lambda \in \bC$ is the moduli space of 
irreducible $\lambda$-connections $(E, \nabla_\lambda)$.
In particular, $\mM_1$ can be identified with $\mM_{dR}$
and $\mM_0$ with the Hitchin moduli space $\mM_H$ of $G$-Higgs bundles 
$(E, \phi)$ on $X$.
The Hodge moduli space $\mM_{Hod}$,
which is the restriction of the twistor space to $\bC \subset \bP^1$,
provides an interpolation between these moduli spaces 
and hence a framework for the heurestic that 
a holomorphic connection is a deformation of a Higgs bundle.

In this paper, for $G=SL_2(\bC)$, 
we will describe $\lambda$-connections $(E,\nabla_\lambda)$ and the associated moduli 
spaces using sub-line bundles $L$ of the underlying rank-2 bundles $E$. 
Central to our consideration will be the moduli space $\N$ of pairs $(E, L)$,
where the rank-2 bundles $E$ have trivial determinant and $ \deg(L) = d$ is fixed.   
There is a natural forgetful map from $\N$ to the moduli space $\Nstable$
of rank-2 bundles with trivial determinant.
In the range $0 < -2d \leq g-1$, the generic fibers of the forgetful map from $\N$ 
to the moduli space $\Nstable$ of stable bundles are finite,
making $\N$ a very useful auxiliary space to explicitly investigate $\Nstable$ 
and more generally the Segre strata on $\Nstable$.

We will consider bundles on $\N$ that capture important data 
of triples $(E, L, \nabla_\lambda)$ in this paper.
For $\lambda = 0$, i.e. the case of Higgs bundles,
it is the cotangent space $T^\ast \N$ that captures important data of triples $(E, L, \phi)$. 
For example, for $-2d = g-1$, away from the loci where $E$ is unstable 
or has nilpotent Higgs fields with kernel $L$, we have $T^\ast \N$ 
as the moduli space of triples $(E, L, \phi)$ with $E$ stable. 
In \cite{DT1},  
the author with collaborator have constructed a rational map
\begin{align}
	&\SOV: T^\ast \N \dashrightarrow (T^\ast X)^{[m]},
	&m=2g-2 -2d.
\end{align}
\begin{theorem} \cite{DT1} \label{thm-SOV-Higgs}
	$\SOV$ is a dominant Poisson map with respect to the natural 
	Poisson structure of cotangent bundles.
\end{theorem}
\noindent
The map $\SOV$ generalises Sklyanin's Separation of Variables approach 
for the Gaudin model, which can be regarded as a variant of the Hitchin moduli space 
-- this explains our notation for the map.
This map can also be regarded as the classical limit of Drinfeld's approach 
to the geometric Langlands correspondence \cite{Drinfeld, Fre95}.

\paragraph{Main results}
The main goal of this paper is to construct the generalisation of $\SOV$ and 
Theorem \ref{thm-SOV-Higgs} for $\lambda$-connections. We will consider 
an affine bundle $\MdRd$ on $\N$ modeled over $T^\ast \N$;
it is the analogue of the restriction 
to $\Nstable$ of the moduli space $\mM_\lambda$, which is an affine bundle
modeled over $T^\ast \Nstable$.
The main result of this paper is an affine analogue of Theorem \ref{thm-SOV-Higgs}.
\begin{theorem}\label{thm-SOV-dR-1}
    There exists a rational dominant
    Poisson map with respect to natural Poisson structures
    \begin{align}
        &\SOVdR: \MdRd \dashrightarrow \Mop,
        &m = 2g-2 - 2d,
    \end{align}
    where $\Mop$ is the symmetric product of an affine bundle 
    modeled over $T^\ast X$.
\end{theorem}

The construction of $\SOVdR$ and the proof that it is a Poisson map are analogous 
to those of $\SOV$.
For the latter, 
the key ingredients are constructions of certain effective divisors
on smooth spectral curves of Higgs bundles $(E, \phi)$
from triples $(E, \phi, L)$, which induce points in $T^\ast \N$.
The authors called these effective divisors Baker-Akhiezer divisors in \cite{DT1}.
The analogues of these divisors are the positions and residue parameters of 
simple apparent singularities of \textit{branched projective structures}, 
or equivalently \textit{opers} with gauge group $PSL_2(\bC)$ on $X$.
The map $\SOVdR$ encodes the positions and residue parameters 
of apparent singularities of opers that are induced from triples $(E, L, \nabla_\lambda)$.
In fact, for $-2d = g-1$,
an open dense subset of $\Mop$ parametrises opers 
on $X$ with $m = 3g-3$ simple apparent singularities 
at which no quadratic differential vanishes.
This explains the notation we have used for this affine bundle. 
The parametrisation is done through the 
positions and residue parameters of the apparent singularities.

We note that the idea of parametrising opers 
using apparent singularities and residue parameters 
was used by Iwasaki 
in the case of Riemann surfaces of arbitrary genus 
with punctures (i.e. these opers have non-apparent singularities) \cite{Iwa91}.
For the case $g=0$, 
Oblezin has studied a map similar to $\SOVdR$
in the context of isomonodromic deformation \cite{Oblezin};
see also \cite{MOA22, MA23} for recent works that emphasise explicit calculations.
Theorem \ref{thm-SOV-dR-1} hence extends Oblezin's result to the case of 
Riemann surfaces of arbitrary genus without punctures.
The use of sub-line bundles $L$ of $E$ 
and the affine bundle $\MdRd$ on the moduli space $\N$ 
are very natural from this point of view.

We also note that in this paper, the Poisson structures in the main theorem \ref{thm-SOV-dR-1} 
are natural to the constructions of the domain and target spaces as affine bundles. 
One can also consider the symplectic structures defined by 
pulling-back the natural symplectic structure of the restriction of $\mM_{dR}$ 
to $\Nstable$ along $\N \dashrightarrow \Nstable$ on one hand, and  
the natural symplectic structure on the monodromy representation variety 
$\mathrm{Hom}(\pi_1, PSL_2(\bC))$ /\hspace{-1pt} $\sim$ \hspace{2pt} along the monodromy map on the other hand.
It is natural to expect that the Poisson structures in our paper 
are indeed symplectic structures and coincide with those defined by pull-back.
In the punctured Riemann surface setting, 
these identities were proved by Pinchbeck \cite{Pinch} 
and Iwasaki \cite{Iwa92} respectively. 
We also note that Pinchbeck has discussed an idea similar to Theorem \ref{thm-SOV-dR-1}.

\paragraph{Structure of the paper}
We start by, in Section 2, discussing the characterisation of $\lambda$-connections 
using sub-line bundles. 
In particular, we construct the affine bundle $\MdRd$ by explicit 
affine coordinate transformations,
and discuss how 
it is essentially a symplectic reduction of an affine bundle $\widetilde{\MdRd}$ 
modeled over $T^\ast \M$,
where $\M$ is the moduli space of embeddings $(L \hookrightarrow E)$.
This is the analogue of the fact that $T^\ast \N$ is essentially a symplectic reduction 
of $T^\ast \M$. 

In Section 3, we discuss the notion of opers with apparent singularities 
and how the transformation rules of the associated residue parameters 
define an affine bundle modeled over $T^\ast X$. 
The $m$-fold symmetric product $\Mop$ of this affine bundle parametrises 
opers with $m$ simple apparent singularities up to a choice of quadratic differential
vanishing at these singularities. This follows from the characterisation of opers 
in terms of the positions and residue parameters of apparent singularities.

Section 4 is for the construction of $\SOVdR$ and the proof of Poisson property.
We show how the projectivisation of the data $(E, L, \phi)$ defines 
an oper with apparent singularities.
We recall the construction of Baker-Akhiezer (BA) divisors from triples $(E,L,\phi)$
from \cite{DT1} to show that 
the induced apparent singularities and residue parameters are indeed 
analogues of BA divisors. 
The map $\SOVdR$ which encode these data is  hence a generalisation of 
$\SOV$. The proof of Poisson property of $\SOVdR$ 
is analogous to that of $\SOV$.

In Section 5,
we discuss the relation between wobbly bundles, 
i.e. stable bundles admitting nonzero nilpotent Higgs fields,
and the loci in $\mM_{\mathrm{op}, \lambda}^{3g-3}$
that needs to be removed to have a regular moduli space of opers.
We then discuss a Lagrangian subspace of $\mM_H$ that is mapped via $\SOV$
to a Lagrangian in $(T^\ast X)^{[3g-3]}$ defined by fixing a divisor on $X$, 
and the analogue of this in $\mM_\lambda$.
These subspaces to some extent resemble the 
Lagrangian leaves in $\mM_H$ and $\mM_\lambda$ 
associated to $\bC^\ast$-fixed point $(E,\phi) \in \mM_H^{\bC^\ast}$
for $E$ unstable \cite{S08}.

\section{Rank-2 $\lambda$-connections and sub-line bundles}
In this section, we discuss how $SL_2(\bC)$ $\lambda$-connections
$(E, \nabla_\lambda)$ and the associated moduli spaces
can be understood more explicitly if we consider them together 
with subbundles $L$ of degree $d$ of $E$. 
In particular, we construct a moduli space $\MdRd$ that 
captures important aspects of the triples $(E, L, \nabla_\lambda)$.
To this end, we start by reviewing an explicit description of the 
moduli space $\N$ of pairs $(E, L)$ and its cotangent space $T^\ast \N$,
which is the analogue of $\MdRd$ for triples $(E, L, \phi)$.

\subsection{Moduli space of pairs $(E, L)$ and its cotangent spaces}
We denote by $\M$ the moduli space of nowhere-vanishing morphisms 
$L \hookrightarrow E$ where $L$ is a line bundle of degree $d$ on $X$ and $E$ 
a rank-2 bundle of trivial determinant.
Alternatively, a point in $\M$ is an equivalence class of extensions 
of the form 
\begin{equation}\label{extension-class}
    0 \longrightarrow L \longrightarrow E \longrightarrow L^{-1} \longrightarrow 0.
\end{equation}
The forgetful map that picks out $L$ 
makes $\M$ a vector bundle over the Picard component $\Pic^d$
parametrizing line bundles of degree $d$ on $X$, with the fiber over 
$L \in \Pic^d$ the space $\mathrm{Ext}(L^{-1}, L) \simeq H^1(X, L^2)$ 
of extension classes of $L^{-1}$ by $L$.
We will abuse the notation by denoting by $(L \hookrightarrow E)$ for both the point in 
$H^1(X, L^2)$ and the point in $\M$ defined by 
an extension of the form \eqref{extension-class}.
Later when we equip coordinates $\bmx$ for $H^1(X, L^2)$, 
we will also denote these points by $\bmx$.

The projectivisation of $\M$ is the moduli space $\N$ of pairs $(E, L)$ where $L$ 
is a subbundle of $E$. It is a projective fiber bundle over 
$\Pic^d$ with fiber $\bP H^1(L^2)$ over $L$. 
The diagram
\begin{equation}
\begin{tikzcd}
    &\M \arrow[dashed]{dr}[swap]{I} 
    \arrow{rr}{\mathrm{pr}} & &\N \arrow[dashed]{dl}{i}  \\
    & &\Nstable &
\end{tikzcd}
\end{equation}
is commutative, where $\mathrm{pr}$ is the projectivisation map 
that sends $(L \hookrightarrow E)$ to the class $[L \hookrightarrow E] \equiv (E, L)$,
and $I$ and $i$ are forgetful maps that pick out only isomorphism classes of $E$.

\paragraph{Coordinates on $\M$} 
Given $(L \hookrightarrow E) \in \M$,
one can introduce local coordinates on a neighborhood of this point 
by fixing some reference divisor on $X$
and describe extension classes of line bundles explicitly. 
The general idea goes back to Tyurin \cite{Tyurin},
while the specific introduction of these coordinates used in this paper 
was discussed also in \cite{Pinch}.
We refer to our paper \cite{DT1} for a more complete discussion.

First, recall that there is an embedding of $X$ in 
$\bP H^1(L^2)$ via the linear system associated to $KL^2$.
We then choose a divisor $\mathbf{p} = \sum_{r=1}^N p_r$ on $X$, 
where $N = h^1(L^2)$,
such that spanning the embedding of $p_r$ gives us $\bP H^1(L^2)$. 
A generic divisor on $X$ would satisfy this condition. 
Next, choose a divisor $\bqcheck = \sum_{i=1}^{g- d} \check{q}_i$ 
such that the divisor $\bq = \sum_{i=1}^g q_i$ 
satisfying $L \simeq \mathcal{O}_X(\bq - \bqcheck)$ is not exceptional, and that 
$\mathbf{q} + \mathbf{p} + \mathbf{\check{q}}$ has no point with multiplicity $ > 1$. 
A divisor $\mathbf{p} + \mathbf{\check{q}}$ satisfying these conditions
will be referred to as a reference divisor 
for $(L \hookrightarrow E) \in \M$.

Let us now also choose local coordinates $\underline{z}_i$, $\underline{\check{z}}_j$ and $w_r$ 
on neighborhoods of $q_i$, $\check{q}_j$ and $p_r$ respectively.
Let $z_i = \underline{z}_i(q_i)$
and assume that $\underline{\check{z}}_i(\check{q}_i) = 0 = w_r(p_r)$.
Let $X_{\bq}$ be the complement of $\mathrm{supp}( \bq + \bp + \bqcheck)$.
Then the extension class $(L \hookrightarrow E) \in \M$ 
can be represented by transition functions 
\begin{align}\label{extension-data}
		&\begin{pmatrix} \underline{z}_i - z_i & 0 \\ 0 & \left( \underline{z}_i - z_i \right)^{-1}	\end{pmatrix},
		& &\begin{pmatrix} \underline{\check{z}}_j^{-1} & 0 \\ 0 & \underline{\check{z}}_j \end{pmatrix},
		& &\begin{pmatrix} 1 & x_r w_r^{-1} \\ 0 & 1 	\end{pmatrix},
\end{align}
of $E$ when transiting from $X_{\bq}$ to neighborhoods of points in $\mathbf{q} + \mathbf{p} + \mathbf{\check{q}}$. 
Here $x_r$ are complex numbers and hence are coordinates on $H^1(L^2)$;
we are able to represent $(L \hookrightarrow E)$ by such transition functions 
thanks to our assumption on $\bp$.
Thanks to our assumption on $\bq$ and $\bqcheck$,
varying $\bq$ then defines a neighborhood of $L$ in $\Pic^d$.
In conclusion, upon choosing local coordinates  
and reference divisor on $X$, 
we have local coordinates $\bmz = (z_1(q_1), \dots, z_g(q_g))$ 
on $\Pic^d$ and $\bmx = (x_1, \dots, x_N)$ on the fibers of $\M$ over $\Pic^d$.

In fact, there is a more natural set of local coordinates on $\Pic^d$ defined via the Abel map.
Let $(\omega_i)_{i=1}^g$ be a basis of holomorphic differentials w.r.t. 
a canonical basis of cycles on $X$.
Then 
\begin{align} \label{Abel-map}
    &\bmz \mapsto A(\bmz) = \blambda = (\lambda_1, \dots, \lambda_g),
    &\lambda_i(\mathbf{q}) \coloneqq \sum_{j=1}^g \int_{x_0}^{q_j} \omega_i 
\end{align}
defines a change of local coordinates on $\Pic^d$.
The differential $\omega_i(z_j(q_j))$ of this change of coordinates is invertible 
as $\bq$ is not exceptional by our assumption.

\paragraph{Higgs fields in terms of meromorphic differentials} 
The reference divisor $\bp + \bqcheck$ allows one to represent Higgs fields 
in terms of Abelian differentials. 
Suppose in local frames adapted to the embedding $\bmx \equiv (L \hookrightarrow E)$ 
over $X_{\bq}$
a Higgs field 
$\phi$ on $E$ takes the form
\begin{equation} \label{phi-abelian-diff}
    \phi \mid_{X_{\bq}} = \begin{pmatrix}
        \phi_0 & \phi_- \\ \phi_+ & \phi_+
    \end{pmatrix}.
\end{equation}
For a divisor $D$ on $X$, let $\Omega_{D}$ be the space of meromorphic differentials
with divisor bounded below by $-D$.
The components $\phi_0$ and $\phi_\pm$ are Abelian differentials on $X$ 
with singular properties depending on the coordinates 
$\bmx = (x_1, \dots, x_N)$ of $(L \hookrightarrow E) \in \M$.

\begin{proposition}\label{prop-abelian-diff-Higgs} \cite{DT1}
	With the setup as above, $\phi_0$, $\phi_\pm$ satisfy
	\begin{itemize}
		\item[(i)] $\phi_+$ is an element of $\Omega_{-2\bq + 2\bqcheck} \simeq H^0(X, KL^{-2})$;
		\item[(ii)] $\phi_0$ is an element of $\Omega_{\bp}$ with $\underset{p_r}{\Res}
        \hspace{2pt} \phi_0 = -x_r \phi_+(p_r)$ at each $p_r$;
		\item[(iii)] $\phi_- = (-\det(\phi) - \phi_0^2)/\phi_+$ 
        is an element of $\Omega_{2\bp + 2\bq - \bmr}$,
		with the singular parts at each $p_r$ fully determined in terms of $\bmx$, $\phi_0$ and $\phi_+$.
	\end{itemize}
\end{proposition} 

The Abelian differential $\phi_+$ can be identified with the composition
\begin{equation}\label{lower-left-Higgs}
    c_{L\hookrightarrow E}(\phi) \equiv c_\bmx(\phi): L \hookrightarrow E \overset{\phi}{\rightarrow} 
    E \otimes K \rightarrow L^{-1}
\end{equation}
and defines an element of $H^0(X, KL^{-2})$, 
which by Serre duality is dual to the fiber $H^1(X, L^2)$.
When the scaling of the embedding $L \hookrightarrow E$ is not important,
we will simply write $c(E, L, \phi)$.
The following condition follows immediately from the singular properties of $\phi_0$.
\begin{corollary} \label{cor-Serre-duality-0} \cite{DT1}
    Given a fixed extension class $(L \hookrightarrow E) \equiv \bmx \in H^1(X, L^2)$
    and a Higgs field $\phi$ on $E$, 
    $c_\bmx(\phi)$ is contained in the hyperplane 
    $\ker(\bmx) \subset H^0(X, KL^{-2})$, namely
    $$\langle c_\bmx(\phi), \bmx \rangle = 0.$$
\end{corollary}
The space of all elements $c_\bmx(\phi)$ 
is the image of the map 
to $H^0(X, KL^{-2})$ in the long exact sequence
\begin{equation}\label{l.e.s.}
    0 \longrightarrow H^0(X, E^\ast LK) \longrightarrow H^0(X, \mathrm{End}_0(E) \otimes K) 
    \overset{c_\bmx}{\longrightarrow} H^0(X, K L^{-2}) \longrightarrow H^1(X, E^\ast LK) \longrightarrow \dots
\end{equation}
which is induced by the embedding of the bundle $E^\ast LK$ 
of $L$-invariant Higgs fields on $E$
to the bundle of trace-free Higgs fields 
$\mathrm{End}_0(E) \otimes K$.

\begin{proposition}\label{prop-surjection-c-Higgs}
If $E$ is stable and there exists a unique embedding $L \hookrightarrow E$ 
up to scaling, then $\mathrm{im}(c_\bmx) = \ker(\bmx)$.
In particular, if $L$ is a subbundle of maximal degree in $E$ then
$\mathrm{im}(c_\bmx) = \ker(\bmx)$.
\end{proposition}
\begin{proof}
The proof follows from Riemann-Roch computation and exactness of the sequence 
\eqref{l.e.s.}. Indeed, 
the dimension of $\mathrm{im} (c_{\bmx})$ 
is $h^0(X, \mathrm{End}_0(E) \otimes K) - h^0(X, E^\ast LK_X)$, 
which for $E$ stable has codimension $h^0(X, L^{-1} E)$.
The second statement follows since a maximal subbundle has a unique 
embedding up to scaling \cite{Lange-Narasimhan}.
\end{proof}
In the range $1\leq -2d  = -2 \deg(L) \leq 2g-1$,
for a generic $L$ and extension $\bmx \in H^1(X, L^2)$, 
we have $h^0(X, L^{-1}E) = 1$.
Proposition \ref{prop-surjection-c-Higgs} says that the generic situation in this range
allows us to pick any element in 
the hyperplane $\ker(\bmx)$ and find a corresponding lower-left component $\phi_+$ of 
some Higgs fields on $E$.

\paragraph{Symplectic structures on $T^\ast \M$ and $T^\ast \N$}
The cotangent bundle $T^\ast \M$ carries a canonical symplectic structure. 
Let $\bmzcheck^H = (\check{z}^H_1, \dots, \check{z}^H_g)$, 
$\bkappa^H = (\kappa^H_1, \dots, \kappa^H_g)$ and 
$\bmk^H = (k^H_1, \dots, k^H_N)$ be the conjugate coordinates 
\footnote{The upper script ``$H$'' is meant to suggest that these Darboux coordinates 
on $T^\ast \M$ can be regarded as coordinates of (pull-backs of) Higgs bundles. 
We will use notations without this upper script for Darboux coordinates 
in the $\lambda$-connections setting (cf. subsection \ref{sect-affine-bundles}).
}
to $\bmz$, $\blambda$ and $\bmx$ respectively, i.e. 
the canonical symplectic form on $T^\ast \M$ takes the local form
\begin{equation}\label{symplectic-form-T*M}
    \tilde{\omega} = \sum_{i=1}^g dz_i \wedge d\check{z}^H_i 
    + \sum_{r=1}^N dx_r \wedge dk^H_r 
    = \sum_{i=1}^g d\lambda_i \wedge d\kappa^H_i 
    + \sum_{r=1}^N dx_r \wedge dk^H_r.
\end{equation}
One key step to the main results in \cite{DT1} is the relation between Darboux coordinates 
on $T^\ast \M$ and Abelian differentials representing Higgs fields.
We now recall this result. 
For a point $\bmx = (L \hookrightarrow E) \in \M$
with $E$ stable,
we will use the short-hand notations $\check{z}^H_i (\phi)$, $k^H_r(\phi)$ 
and $\kappa^H_i (\phi)$ for the respective coordinates of 
the pull-back $I^\ast(\phi) \in T^\ast_{\bmx} \M$ 
of a Higgs field $\phi$ on $E$.

\begin{proposition}\label{prop-coordinate-Higgs} \cite{DT1}
Suppose a Higgs field $\phi$ on $E$ takes the local form 
$\big(\begin{smallmatrix} \phi_0 & \phi_- \\ \phi_+ & - \phi_0 \end{smallmatrix}\big)$
on $X_{\bq}$ as in \eqref{phi-abelian-diff}. 
Then
\[
\begin{aligned}
\check{z}^H_i (\phi) &= - 2 \phi_0( q_i), \quad i=1,\dots,g,\\
k^H_r(\phi) & = \phi_+(p_r),\qquad r=1,\dots,N,
\end{aligned}\qquad
\kappa^H_i (\phi) = - 2 \sum_{j=1}^g (dA^{-1})_{ij} \phi_0(q_j).
\]
where the evaluations of Abelian differentials are w.r.t. chosen local coordinates.
\end{proposition}

There is a canonical $\bC^\ast$-action on $\M$ which 
acts along the fibers over $\Pic^d$ and scales the extension classes.
The induced action on $T^\ast \M$ is defined by pull-back. 
In terms of Darboux coordinates,
\begin{align}\label{Cstar-action}
    &\epsilon. (\bmx, \blambda, \bmk^H, \bkappa^H) 
    = (\epsilon \bmx, \blambda, \epsilon^{-1} \bmk^H, \bkappa^H),
    &\epsilon \in \bC.
\end{align}
The corresponding moment map $H: T^\ast \M \rightarrow \bC$ 
is nothing but the Serre duality pairing 
of $(L \hookrightarrow E) \in H^1(X, L^2)$ 
with $c_\bmx(\phi) \in H^0(X, KL^{-2})$. 
In terms of Darboux coordinates,
\begin{equation}\label{moment-map}
    H\left( (\bmx, \blambda, \bmk^H, \bkappa^H) \right)
    = \bmx . \bmk^H = \sum_{r=1}^N x_r k^H_r.
\end{equation}
By Corollary \ref{cor-Serre-duality-0}, the pull-back of a Higgs field $\phi \in T^\ast_E\Nstable$ 
to $T^\ast_{(L\hookrightarrow E)} \M$
is contained in the level set $H^{-1}(0)$.
These pull-backs are $\bC^\ast$-equivariant and commute with the 
pull-backs to $T^\ast_{[L\hookrightarrow E]} \N$.
In fact, an open dense subset of $T^\ast \N$ is symplectomorphic to 
the quotient of an open dense subset of $H^{-1}(0)$ by the $\bC^\ast$-action, 
i.e. a symplectic reduction of an open dense subset of $T^\ast \M$
(cf. the Appendix B in \cite{DT1}).

\paragraph{Fibers of $\N \dashrightarrow \Nstable$}
The forgetful map that picks out only $E$ defines rational maps from $\M$ 
and $\N$ to the moduli space $\Nstable$ of rank-2 stable bundles with trivial determinant. 
Denote by $i$ the map $\N \dashrightarrow \Nstable$.
For $1 \leq -2d \leq g-1$, a dense subspace of the moduli space $\N$ is defined by
pairs $(E, L)$ with $L$ a subbundle of maximal degree of $E$. 
The bundle $E$ in this case is stable 
and lies in the stratum 
$$S_{d} \coloneq \left\{ E \in \Nstable \mid  \underset{L \subset E}{\max} \deg(L) \geq d   \right\}$$ 
of the Segre stratification on $\Nstable$ \cite{Lange-Narasimhan}. 
In particular, $\Nstable = S_d$ for $-2d = g-1$.
The following results are well-known.

\begin{proposition} \label{prop-finite-covering}
\begin{enumerate}
    \item For $-2d = g-1$, there is an open dense subset of $\Nstable = S_d$ 
    elements of which admit finitely many maximal subbundles \cite{Lange-Narasimhan}. 
    In particular, very-stable bundles, which are those that admit no nonzero nilpotent 
    Higgs fields and define an open dense subset of $\Nstable$, 
    admit exactly $2^g$ maximal subbundles \cite{Gronow}.
    \item For $1 \leq -2d \leq g-2$, there is an open dense subset of $S_d$ 
    elements of which admit exactly 1 maximal subbundle \cite{Lange-Narasimhan}.
\end{enumerate}
\end{proposition}

It follows from Proposition \ref{prop-finite-covering} that for $1 \leq -2d \leq g-1$, 
there is a dense subset $\N^{\mathrm{im}} \subset \N$ 
consisting of pairs $(E, L)$ 
where $E$ is stable and the restriction 
of the forgetful map $i: \N \dashrightarrow \Nstable$ 
to a neighborhood of $(E, L)$ is an immersion 
(and an isomorphism if $-2d = g-1$). 
We denote by $\Nstable^{\mathrm{im}}$ the image of $\N^{\mathrm{im}}$.

\begin{proposition}\label{prop-immersion}
    For $1 \leq -2d \leq g-1$, 
    the map $i$ is not an immersion at $[\bmx] \equiv (E, L)$ 
    if and only if 
    the line bundle $K_X L^2$ defines an exceptional divisor on $X$, 
    i.e. $h^0(X, K_X L^2)$ is larger than its expected value. 
\end{proposition}
\begin{proof}
    It follows from Proposition \ref{prop-coordinate-Higgs} that 
    the kernel of $i^\ast_{(E,L)}: T^\ast_E \Nstable \rightarrow T^\ast_{(E,L)}\N$ 
    is isomorphic to the space $H^0(X, K_X L^2)$
    of nilpotent Higgs fields on $E$ with kernel $L$.
    But the expected value of $h^0(X, K_X L^2)$ is $g-1 + 2d$, 
    which is equal to $\dim(\Nstable) - \dim(\N)$.
\end{proof}

It follows from Proposition \ref{prop-immersion} that, for $-2d = g-1$,
an open dense subset of $T^\ast \N$ is 
the moduli space of triples $(E, L, \phi)$ with $E$ stable
(modulo the loci of unstable bundles 
and the loci $(E, L) \in \N$ where $E$ has nilpotent Higgs fields 
with kernel $L$).
For $1\leq -2d < g-1$, Propositions \ref{prop-coordinate-Higgs} 
and \ref{prop-immersion} essentially say that, 
away from the loci $(E, L) \in \N$ where $E$ has 
more nilpotent Higgs fields with kernel $L$ than expected, 
$T^\ast \N$ is the moduli space of the ``lower-triangular part'' of the 
Higgs fields in local frames adapted to $L$.

\subsection{$\lambda$-connections in terms of Abelian differentials}
Let us fix $\lambda \in \bC$.
An $SL_2(\mathbb{C})$ $\lambda$-connection is a pair $(E, \nabla_\lambda)$ 
where $E$ is a holomorphic bundle of trivial determinant 
and $\nabla_\lambda$ is a holomorphic map $\nabla_\lambda: E \rightarrow E \otimes K_X$ 
such that
\begin{itemize}
    \item $\nabla_\lambda$ induces the trivial connection scaled by 
    $\lambda$ on $\mathcal{O}_X$;
    \item $\nabla_\lambda$ satisfies a twisted Leibniz rule 
    $\nabla_\lambda(fs) = \lambda \hspace{2pt} df \otimes s + f \nabla(s)$  
    for all local holomorphic functions $f$ and holomorphic sections $s$ of $E$.
\end{itemize}
Note that the case $\lambda = 0$ reduces to that of Higgs bundles. 
Note also that given a fixed holomorphic bundle $E$, the space of $\lambda$-connections 
on $E$ is an affine space modeled over the space 
$H^0(X, \mathrm{End}_0(E) \otimes K_X)$ 
of $SL_2(\bC)$-Higgs fields.

The following results concerning $\lambda$-connections 
can be regarded as the generalisation of 
results for Higgs bundles that we have reviewed to the case $\lambda \neq 0$.
Given $(L \hookrightarrow E) \equiv \bmx \in H^1(X, L^2)$ 
and a $\lambda$-connection $(E, \nabla_\lambda)$, the composition
\begin{equation}\label{lower-left}
    c_\bmx(\nabla_\lambda): L \longhookrightarrow E \overset{\nabla_\lambda}{\longrightarrow} E \otimes K_X \longrightarrow L^{-1} K_X   
\end{equation}
is $\mathcal{O}_X$-linear and defines a section of $KL^{-2}$. 
It can be regarded as the generalisation of $c_\bmx(\phi)$ 
defined in \eqref{lower-left-Higgs} for the $\lambda$-connection setting. 
When the scaling of the embedding $L \hookrightarrow E$ is not important, 
we will also write $c(E, L, \nabla_\lambda)$.
The generalisation of Proposition \ref{prop-abelian-diff-Higgs} is as follows.
Suppose $\nabla_\lambda$ takes the form
\begin{equation}\label{Abelian-differentials}
   \nabla_\lambda \mid_{X_{\bq}} = \lambda \partial + 
\left(\begin{matrix}  \omega_0 & \omega_- \\ \omega_+ & -\omega_0  \end{matrix}\right) 
\end{equation}
in some frames adapted to $L$ over $X_{\bq}$.
Then $\omega_0$ and $\omega_\pm$ are Abelian differentials 
that are holomorphic on $X_{\bq}$ and satisfy the following conditions. 

\begin{proposition} \label{prop-abelian-differentials}
With the above setup, $\omega_0$, $\omega_-$ and $\omega_+$ are Abelian differentials satisfying the 
following properties:
\begin{itemize}
    \item $\omega_+ \in \Omega_{-2 \mathbf{q} + 2\bqcheck}$ can be identified 
    with $c_\bmx(\nabla_\lambda)$. 
    \item $\omega_0 \in \Omega_{\bq + \bqcheck + \bp}$ with residues 
    \begin{equation} \label{residue-omega0}
        \underset{y}{\Res} \hspace{2pt} \omega_0 = 
        \begin{cases}
            - \lambda & \text{ for } y \in \bqcheck, \\
            \lambda & \text{ for } y \in \bq, \\
            -x_r \omega_+(p_r) & \text{ for } y \in \bp.
        \end{cases}
    \end{equation}
    \item $\omega_- \in \Omega_{2\bp + 2 \bq - 2 \bqcheck}$ with the singular part at each $p_r$ 
    fully determined by $\omega_0$, $\omega_+$ and $\lambda$.
\end{itemize}
\end{proposition}
\begin{proof}
    Using the transition functions \eqref{extension-data},
    one obtains the local form $\lambda \partial + A$ 
    of $\nabla_\lambda$ in neighborhoods of 
    points in $\bp +  \bq + \bqcheck$ where $A$ 
    is the local 1-form
    \begin{subequations}\label{local-1-form}
    \begin{align}
        &\begin{pmatrix}
            \omega_0 + \lambda\underline{\check{z}}_i^{-1} & \underline{\check{z}}_i^{-2} \omega_- \\
            \underline{\check{z}}_i^2 \omega_+ & - \omega_0 - \lambda\underline{\check{z}}_i^{-1}
        \end{pmatrix}& &\text{and} &
        &\begin{pmatrix}
            \omega_0 - \lambda(\underline{z}_j - z_j)^{-1} 
            & (\underline{z}_j - z_j)^{2} \omega_- \\
            (\underline{z}_j - z_j)^{-2} \omega_+ & - \omega_0 + \lambda(\underline{z}_j - z_j)^{-1}
        \end{pmatrix}
    \end{align}
    around $\check{q}_i$ and $w_r$ respectively, and 
    \begin{equation}
        \begin{pmatrix}
            \omega_0 - x_r \omega_+ w_r^{-1}
            & \omega_- - 2x_r \omega_0 w_r^{-1} - x_r^2 \omega_+ w_r^{-2} + \lambda x_r w_r^{-2} \\
            \omega_+ &-\omega_0 + x_r \omega_+ w_r^{-1}
        \end{pmatrix}
    \end{equation}
    \end{subequations}
    around $p_r$.
    The proof now follows from regularity of $A$ at $\bp +  \bq + \bqcheck$.
\end{proof}

The vanishing of sum of residues of $\omega_0$ immediately implies the following.
\begin{corollary}\label{cor-Serre-lambda}
    The Serre duality pairing of $c_\bmx(\nabla_\lambda)$ and 
    $\bmx \equiv (L \hookrightarrow E)$ is 
    \begin{equation}\label{Serre-constraint}
        \langle c_\bmx(\nabla_\lambda), \bmx \rangle 
        = \lambda d.
    \end{equation}
\end{corollary}

The pairing \eqref{Serre-constraint} is the generalisation of 
the vanishing pairing in Corollary \ref{cor-Serre-duality-0} to $\lambda \neq 0$.
Note that if $E$ is strictly semi-stable with destablising subbundle $L$ 
then the pairing \eqref{Serre-constraint} also vanishes.
The following Proposition, which is the affine analogue of Proposition 
\ref{prop-surjection-c-Higgs}, follows immediately.

\begin{proposition}\label{prop-surjection-c-2}
    If $E$ is stable and $L$ has a unique embedding into $E$ up to scaling, 
    then the projection from the 
    space of $\lambda$-connections on $E$ 
    to the affine hyperplane in $H^0(X, K_X L^{-2})$
    defined by \eqref{Serre-constraint} is surjective.
    In other words, in such a situation, 
    given any element $c \in H^0(X, KL^{-2})$ satisfying 
    $\langle c, \bmx \rangle = \lambda d$, 
    then there exists a $\lambda$-connection 
    $(E, \nabla_\lambda)$ 
    such that $c = c_\bmx(\nabla_\lambda)$.
\end{proposition}

\subsection{Symplectic affine bundles on $\N$ and $\M$}
\label{sect-affine-bundles}
\paragraph{Affine bundle on $\M$}
We now define an affine bundle $\widetilde{\MdRd}$ on $\M$ modeled over $T^\ast \M$.
It will turn out that this affine bundle $\widetilde{\MdRd}$
is an analogue of the Hodge moduli space $\MHodge$, 
with the moment map $H$ defined in \eqref{moment-map} 
the analogue of the twistor parameter $\lambda$.

The idea of construction is rather straightforward. 
Recall the moment map $H: T^\ast \M \rightarrow \bC$,
$(\bmx, \blambda, \bmk, \bkappa) \mapsto \bmx.\bmk$,
of the $\bC^\ast$-action.
For a fixed $\lambda \in \bC^\ast$, 
consider the affine bundle $\mM^s_\lambda \rightarrow \Nstable$ 
which is the moduli space 
of irreducible $\lambda$-connections with stable underlying bundles.
It follows from Corollary \ref{cor-Serre-lambda} that the 
pull-back along $I: \M \dashrightarrow \Nstable$ of $\mM^s_\lambda$ 
defines an affine bundle on $\M$ with fibers identified with 
$H^{-1}(\lambda d) \cap T^\ast_{(L \hookrightarrow E)}\M$. 
By varying $\lambda$, we can define a bundle with fibers 
identified with the complement of $H^{-1}(0)$ in $T^\ast_{(L \hookrightarrow E)}\M$.
The extension of the associated affine isomorphisms to $H^{-1}(0)$ then defines 
an affine bundle modeled over $T^\ast \M$.

In the following, we are going to explain this construction by 
first introducing the generalisation of the Darboux coordinates 
$\bmk^H$, $\bmzcheck^H$ and $\bkappa^H$ on $T^\ast \M$
to the $\lambda \neq 0$ setting (cf. Proposition \ref{prop-coordinate-Higgs}).
Given $(L \hookrightarrow E) \in \M$ and an 
irreducible $\lambda$-connection $\nabla_\lambda$ on $E$, 
the data consisting of a reference divisor $\bmr = \bp + \bqcheck$ 
and local coordinates around $\bmr + \bq$ let us define the 
vectors $\bmk = (k_1, \dots, k_N) \in \bC^N$ and 
$\bmzcheck = (\check{z}_i, \dots, \check{z}_g) \in \bC^g$ where
\begin{subequations}\label{affine-tuples}
\begin{align}
    &k_r \coloneqq \omega_+(p_r), 
    &\check{z}_i \coloneqq -2 \left( \omega_0 - \frac{\lambda}{\underline{z}_i - 
    z_i} \right)(q_i).
\end{align}
These quantities are well-defined due to regularity of the 1-form \eqref{local-1-form}.
We also define $\bkappa = (\kappa_1, \dots, \kappa_g) \in \bC^g$ where
\begin{align}
    &\kappa_j \coloneqq \sum_{i=1}^g (dA^{-1})_{ji} \check{z}_i,
    &i = 1, \dots, g;
\end{align}
\end{subequations}
here $dA^{-1}$ is the differential of $A^{-1}$ evaluated at $\blambda = A(\bmz)$ (cf. \eqref{Abel-map}). 
By Corollary \ref{cor-Serre-lambda}, these vectors define a map 
$$ I_{\lambda,\bmr}: \mM_{\lambda}\mid_E \longrightarrow H^{-1}(\lambda d) \cap T^\ast_{(L \hookrightarrow E)} \M $$
where $\mM_{\lambda}\mid_E$ is the space of irreducible $\lambda$-connections 
on $E$
and $H^{-1}(\lambda d)$ is the $(\lambda d)$-level set 
of the moment map on $T^\ast \M$ (cf. \eqref{moment-map}).
The image of $I_{\lambda,\bmr}$ is an affine space modeled over 
the image of the pull-back
$$ I^\ast_E: T^\ast \N \longrightarrow H^{-1}(0) \cap T^\ast_{(L \hookrightarrow E)} \M.  $$
It follows from Proposition \ref{prop-immersion} that $I_{\lambda,\bmr}$ is surjective, 
i.e. $E \in \Nstable^{\mathrm{im}}$,
if and only if $h^0(X, KL^2)$ is equal to the expected value $g-1 + 2d$.

Suppose now $E \in \Nstable^{\mathrm{im}}$ and so $I_{\lambda,\bmr}$ is surjective.
Consider a change of reference divisors $\bmr \mapsto \bmr' = \bp' + \bqcheck'$.
In terms of $\bmr'$, let $(L' \hookrightarrow E')$ be the representative of the same point  
in $\M$ defined by $(L \hookrightarrow E)$.
This means $E'$ is defined by some transition functions 
of the form \eqref{extension-data} 
with $\bmr \rightarrow \bmr'$ and $\bq \rightarrow \bq' = q_1' + \dots + q_g'$, 
where $L' = \mathcal{O}_X(\bq' - \bqcheck') \simeq \mathcal{O}_X(\bq - \bqcheck)$ 
is a subbundle of $E'$.
Let $\nabla_\lambda'$ be the $\lambda$-connection on $E'$ corresponding to $\nabla_\lambda$ 
via an isomorphism $E' \simeq E$.
By considering a local form of $\nabla_\lambda'$ 
in local frames adapted to $L'$ over $X\setminus \mathrm{supp}(\bp' + \bq' + \bqcheck')$ 
as in \eqref{Abelian-differentials}, 
one can similarly define a map
$I_{\lambda, \bmr'}: \mM_{\lambda}\mid_{E'} \longrightarrow H^{-1}(\lambda d)$
that associates to $\nabla_\lambda'$ a set of vectors 
$\bmk'$ and $\bmzcheck'$.
For a fixed $\lambda \in \bC^\ast$,
there is a unique affine automorphism of $H^{-1}(\lambda d)$ that relate the 
two sets of tuples $(\bmk, \bmzcheck)$ and $(\bmk', \bmzcheck')$ 
for all irreducible $\lambda$-connections on $E$.
To see this, let us choose an irreducible $\lambda$-connection 
$\nabla_{\mathrm{ref}}$ on $E$, and let
$(\bmk_{\mathrm{ref}}, \bmzcheck_{\mathrm{ref}})$ 
and $(\bmk_{\mathrm{ref}}', \bmzcheck_{\mathrm{ref}}')$
be its associated vectors w.r.t. the two reference divisors.
By Proposition \ref{prop-coordinate-Higgs}, 
the pull-back to $T^\ast\M$ 
of the Higgs field $\nabla_\lambda - \nabla_{\mathrm{ref}}$ has respective 
fiber coordinates
$(\bmk - \bmk_{\mathrm{ref}}, \bmzcheck- \bmzcheck_{\mathrm{ref}})$ 
and $(\bmk' - \bmk_{\mathrm{ref}}', \bmzcheck' - \bmzcheck_{\mathrm{ref}}')$ 
which are contained in $H^{-1}(0)$.
One then obtains a unique affine isomorphism $I_{\lambda, \bmr, \bmr'}$
associated to the transition function of 
$T^\ast \M$ and maps 
$(\bmk_{\mathrm{ref}}, \bmzcheck_{\mathrm{ref}}) \longmapsto 
(\bmk_{\mathrm{ref}}', \bmzcheck_{\mathrm{ref}}')$. 
\begin{equation}\label{affine-structure-1}
	\centering
	\begin{tikzcd}
		&\mM_{\lambda}\mid_E \arrow{r}{I_{\lambda, \bmr}} \arrow{d}{\sim}
        &H^{-1}(\lambda d)  \cap T^\ast_{(L \hookrightarrow E)} \M
         \arrow{d}{I_{\lambda, \bmr, \bmr'}} \arrow[hookrightarrow]{r} 
        &T^\ast_{(L \hookrightarrow E)} \M \arrow{d}{=} \\
        &\mM_{\lambda}\mid_{E'} \arrow{r}{I_{\lambda, \bmr'}} 
        &H^{-1}(\lambda d) \cap T^\ast_{(L' \hookrightarrow E')} \M
        \arrow[hookrightarrow]{r} 
        &T^\ast_{(L' \hookrightarrow E')} \M
	\end{tikzcd}.
\end{equation}
Clearly a different choice of the reference connection 
$\nabla_{\mathrm{ref}}$ induces the same affine isomorphism.

Consider now the Hodge moduli space $\MHodge \rightarrow \bC$ 
of irreducible $SL_2(\bC)$ $\lambda$-connections on $X$. 
The fiber $\mM_0$ over $0$ can be identified with the Hitchin moduli space $\mM_H$,
and over a fixed $\lambda \in \bC^\ast$ 
the moduli space $\mM_{\lambda}$ 
of irreducible $\lambda$-connections $(E, \nabla_\lambda)$.
Varying $\lambda \in \bC^\ast$ and requiring stability of $E$ 
defines an open dense subset $\widetilde{\mM}^s_{\lambda} \subset \MHodge\mid_{\bC^\ast}$ 
which is a vector bundle 
\begin{equation}
\begin{tikzcd}
    \widetilde{\mM}^s_{\lambda} \arrow{d} \\
    \bC^\ast \times \Nstable
\end{tikzcd}
\end{equation}
whose fiber over $\bC^\ast \times \{ E\}$ 
is the space $\widetilde{\mM}^s_{\lambda} \mid_E$ 
of irreducible $\lambda$-connections 
on $E$ modulo isomorphisms for all $\lambda \in \bC^\ast$. 
For each $(L \hookrightarrow E) \in \M$ with $E$ stable, 
upon a choice of reference divisor $\bmr$ and local coordinates,
the vectors defined in \eqref{affine-tuples} define a map 
$$I_\bmr: \widetilde{\mM}^s_{\lambda}\mid_E 
\longrightarrow T^\ast_{(L\hookrightarrow E)}\M \setminus H^{-1}(0).$$ 
A map $I_{\bmr'}$ defined by a different reference divisor $\bmr'$ 
commutes with $I_\bmr$ via the collection of affine automorphisms 
$I_{\lambda, \bmr, \bmr'}$ on the level set $H^{-1}(\lambda d)$ 
as defined in \eqref{affine-structure-1}. 
These collection of affine automorphisms extend to an affine automorphism 
$I_{\bmr, \bmr'}$ of $T^\ast_{(L \hookrightarrow E)}\M$ that restricts to 
the vector space 
isomorphism defined by transition function of $T^\ast \M$ on $H^{-1}(0)$.
Let $\{ \tilde{U}_\alpha \}_{\alpha \in \mathcal{I}}$ 
be an atlas on $\M$ such that 
\begin{itemize}
    \item each $\tilde{U}_\alpha$ can be equipped coordinates 
    using some reference divisor $\bmr_\alpha$;
    \item the images of overlaps $U_\alpha \cap U_\beta$ along $I$
    are contained in $\Nstable^{\mathrm{im}}$.
\end{itemize}
We then can use the affine automorphisms $I_{\bmr_\alpha, \bmr_\beta}$ 
of $T^\ast\M$ to define an affine bundle over $\M$ 
modeled over $T^\ast\M$. 
We denote this affine bundle by $\tMdRd$. 
It follows from \eqref{affine-structure-1} that the restriction of $\tMdRd$ 
to the complement of $H^{-1}(0)$ is simply the pull-back of 
$\widetilde{\mM}^s_{\lambda}$ to $\M$.
Hence $\tMdRd$ plays the role of an analogue of the 
Hodge moduli space $\MHodge$,
with the moment map $H$ encoding the twistor coordinate $\lambda$.

\paragraph{Symplectic structure}
It follows from the local identification of $\tMdRd$ with $T^\ast\M$
that one can define local symplectic form by pulling-back the canonical 
symplectic form on $T^\ast \M$. A priori this local symplectic form on $\tMdRd$ 
is not unique, i.e. it depends on the local identifications. 
Using Proposition \ref{prop-coordinate-Higgs} and \eqref{affine-tuples}, 
we can write this canonical symplectic form as 
\begin{equation}\label{symplectic-form-dR}
    \sum_{i=1}^g dz_i \wedge d\check{z}_i 
    + \sum_{r=1}^N dx_r \wedge dk_r 
    = \sum_{i=1}^g d\lambda_i \wedge d\kappa_i 
    + \sum_{r=1}^N dx_r \wedge dk_r
\end{equation}
in a local neighborhood of $(L \hookrightarrow E)$ with $E\in \Nstable^{\mathrm{im}}$.
It is easy to see that
\begin{align}
    &\sum_{r=1}^N dx_r \wedge dk_r =
    \sum_{r=1}^N dx_r' \wedge dk_r',
    &\sum_{i=1}^g dz_i \wedge d\check{z}_i = 
    \sum_{i=1}^g dz_i' \wedge d\check{z}_i'
\end{align}
upon a change of reference divisor $\bmr \rightarrow \bmr'$. 
Hence we have a global symplectic form $\widetilde{\omega}_{\lambda}$ 
on $\tMdRd$ with local Darboux coordinates $(\bmx, \bmz, \bmk, \bmzcheck)$ 
or alternatively $(\bmx, \blambda, \bmk, \bkappa)$ 
defined through \eqref{extension-data} and \eqref{affine-tuples}
by choosing reference divisors.

We note that for Riemann surfaces with punctures, 
Pinchbeck in \cite{Pinch} considered a similar symplectic structure 
and showed that it coincides with the pull-back of the canonical 
symplectic structure on the de Rham moduli space of holomorphic connections.
We expect by adapting the strategy in \cite{Pinch},
a similar result in the current setting of compact Riemann surfaces can be proved.

\paragraph{Symplectic affine bundle on $\N$}
The restriction of $\widetilde{\mM}^s_\lambda$ to $\{\lambda\} \times \Nstable$
is the moduli space $\mM^s_\lambda$ of irreducible $\lambda$-connections 
with stable underlying bundles.
We denote by $\MdRd$ the pull-back of $\mM^s_\lambda$ along $\N \dashrightarrow \Nstable$.
Clearly $\MdRd$ is an affine bundle modeled over $T^\ast \N$, 
and one can regard this as an analogue of how $\widetilde{\mM}^s_\lambda$ 
is an affine bundle modeled over $T^\ast \Nstable$.

Recall that an open dense subset of $T^\ast\N$ 
is a symplectic reduction of 
an open dense subset of $T^\ast\M$ by $\bC^\ast$-action.
Here,
we also have an open dense subset of $\MdRd$ as a symplectic 
reduction of an open dense subset of $\tMdRd$. 
More precisely, the $\bC^\ast$-quotient of the level set $H^{-1}(\lambda d)$
is isomorphic to an open dense subset in $\MdRd$ as complex manifolds,
such that the pull-back of the local symplectic structure defined by 
local identification with $T^\ast \N$ to $H^{-1}(\lambda d)$ 
coincides with the restriction of $\widetilde{\omega}_\lambda$ to $H^{-1}(\lambda d)$.

\paragraph{Maps from $\mM^s_\lambda$} 
Denote by $\mM^s_\lambda\mid_E$ the fiber of $\mM^s_\lambda$ 
over $E \in \Nstable$: 
in other words, this is the space of isomorphism classes of irreducible $\lambda$-connections 
on $E$.
It is an affine space modeled over $H^0(X, \mathrm{End}_0(X) \otimes K_X) \simeq \bC^{3g-3}$.
It follows from the construction of $\tMdRd$ that,
given $\bmx \equiv (L \hookrightarrow E) \in \M$, the diagram
\eqref{affine-structure-1} defines a map 
$$I_\lambda\mid_{\bmx}: \mM^s_\lambda\mid_E \rightarrow \tMdRd\mid_\bmx$$ 
which upon choosing a reference divisor $\bmr$ 
coincides with $I_{\lambda, \bmr}$.
As we vary $\bmx$ along its $\bC^\ast$-orbit,
the collections of such maps are equivariant w.r.t. the $\bC^\ast$-action 
on $\tMdRd$ and hence descends to a map 
$$i_\lambda\mid_{[\bmx]}: \mM^s_\lambda\mid_E \rightarrow \MdRd\mid_{[\bmx]}$$
where $[\bmx] = (L, E) \in \N$ is the class of $\bmx$ upon modulo scaling.
These two maps are affine analogues of the point-wise pull-backs
\begin{align*}
&I^\ast\mid_\bmx: T_E^\ast \Nstable \longrightarrow T_\bmx^\ast \M,
&i^\ast\mid_{[\bmx]}: T_E^\ast \Nstable \longrightarrow T_{[\bmx]}^\ast \N.
\end{align*}

\begin{proposition} \label{prop-surjectivity-failure}
Let $d$ be in the range $0 <-2d \leq g-1$. 
Let $\nabla_\lambda$ and $\nabla_\lambda'$ be two $\lambda$-connections on $E$, 
and $\bmx \equiv (L \hookrightarrow E) \in \MdRd$.
Then 
\begin{align*}
&I_\lambda(\nabla_\lambda) = I_\lambda(\nabla_\lambda'),
&i_\lambda(\nabla_\lambda) = I_\lambda(\nabla_\lambda')
\end{align*}
if and only if 
$\nabla_\lambda - \nabla_\lambda'$ is a nilpotent Higgs fields with kernel $L$.
In particular, $i_\lambda\mid_{[\bmx]}$ is not surjective if and only if 
$i$ is not an immersion at $[\bmx]$, 
which occurs if and only if $K_x L^2$ defines an exceptional divisor on $X$.
\end{proposition}
\begin{proof}
By construction, we have 
$I_\lambda(\nabla_\lambda) - I_\lambda(\nabla_\lambda') 
= I^\ast\mid_\bmx(\nabla_\lambda - \nabla_\lambda')$.
The statements now follows from Proposition \ref{prop-immersion}.
\end{proof}

It follows that for $-2d = g-1$, the maps $i_\lambda\mid_{[\bmx]}$ 
are isomorphisms on the open dense subset $\N^{\mathrm{im}} \subset \N$.
Consequently, in this case, 
similarly to how the restriction of $T^\ast \N$ to $\N^{\mathrm{im}}$ 
is the moduli space of triples $(E, L, \phi)$ for $(E,L) \in \N^{\mathrm{im}}$, 
we have the restriction of $\MdRd$ to $\N^{\mathrm{im}}$ 
is the moduli space of triples $(E, L, \nabla_\lambda)$ for $(E,L) \in \N^{\mathrm{im}}$.

\subsection{Components of $\lambda$-connections revisited} 
\label{sect-abelian-diff}
We have seen from \eqref{affine-structure-1} and \eqref{symplectic-form-dR}
how the evaluations of Abelian differentials 
$ \left( \begin{smallmatrix} \omega_0 & \omega_- \\ \omega_+ & - \omega_0   
\end{smallmatrix} \right) $ 
in \eqref{Abelian-differentials}
relate to Darboux coordinates of the symplectic form 
$\widetilde{\omega}_\lambda$ on $\widetilde{\mM}^d_\lambda$.
To prepare for the proof of the main theorem in Section \ref{sect-proof},
in the following we express $\omega_0$ in terms of the 
Darboux coordinates.

Given a pair of distinct points $p_\pm$ on $X$, recall that 
the Abelian differential of the third kind $\omega_{p_+ - p_-}$ 
has simple poles at $p_\pm$ with respective residues $\pm 1$ and 
is holomorphic elsewhere. It is the unique meromorphic differential with this 
property with vanishing $A$-cycles. 
Let us now similarly define a unique Abelian differential of the third kind 
$\omega^{\bq}_{p_+ - p_-}$ by imposing a different normalisation condition. 
Namely, we require that $\omega^{\bq}_{p_+ - p_-}$ has simple poles at $p_\pm$ with respective residues $\pm 1$, is holomorphic elsewhere, and in addition 
its 0-th Laurent coefficient w.r.t. the local coordinate $z_i$ 
around each $q_i < \bq$ is $0$. 
In other words, our normalisation condition is that
\begin{subequations}\label{normalise-vanish-q}
\begin{equation}\label{normalise-vanish-q-1}
    \omega^{\bq}_{p_+ - p_-}(q_i) = 0 \text{ for } q_i \notin \{ p_\pm \}
\end{equation}
and 
\begin{equation}\label{normalise-vanish-q-2}
    \omega^{\bq}_{p_+ - p_-} =  \left( \frac{\pm 1}{z_i} + \mathcal{O}(z_i) \right) dz_i \text{ for } q_i \in \{ p_\pm \} .
\end{equation}
\end{subequations}

Note that while condition \eqref{normalise-vanish-q-1} is coordinate-independent, 
the vanishing of the $0$-th Laurent coefficient in \eqref{normalise-vanish-q-2}
is not.
Nevertheless, one can find an explicit relation between 
the two sets of normalised Abelian differentials.
Let $\omega^{0}_{p_+ - p_-}(q_i)$ be the 0-th Laurent coefficient w.r.t. $z_i$
of $\omega_{p_+ - p_-}(q_i)$, namely
\begin{equation*}
    \omega^{0}_{p_+ - p_-}(q_i) 
    = \begin{cases}
    \omega_{p_+ - p_-}(z_i(q_i)) & \text{ for } q_i \notin \{ p_\pm \} \\
    \left( \omega_{p_+ - p_-} \mp \frac{ dz_i}{z_i}  \right) (z_i(q_i)) & \text{ for } q_i \in \{ p_\pm \}.
    \end{cases}
\end{equation*}
Then it is straightforward to check that 
\begin{equation}\label{relations-Abelian-diff}
    \omega^{\bq}_{p_+ - p_-} = \omega_{p_+ - p_-} 
    - \sum_{i, n = 1}^g \left( dA^{-1}_{\bm{\lambda}} \right)_{ni} \omega^{0}_{p_+ - p_-}(q_i) \omega_n.
\end{equation}
Knowing the residues of $\omega_0$ and using the 
definition of $\check{z}_i$ and $\kappa_i$ from \eqref{affine-tuples}, 
we can write
\begin{equation}
    \omega_0 = - \sum_{r=2}^N x_r k_r \omega^{\bq}_{p_r - p_1} 
    + \lambda \sum_{i=1}^g \omega^{\bq}_{q_i - p_1}
    - \lambda \sum_{i=1}^{g-d} \omega^{\bq}_{\check{q}_i - p_1} 
    - \frac{1}{2} \sum_{i=1}^g \kappa_i \omega_i.
\end{equation}
This formula gives a concrete way to express $\omega_0$
in terms of the coordinates on $\tMdRd$.

We can also express $\omega_+$ in terms of its zeroes and poles. 
Suppose $\bm{u} = u_1 + \dots + u_m$ where $m = \deg(KL^{-2}) = 2g-2 - 2d$ 
is its zero divisor. It follows from Proposition \ref{prop-abelian-differentials} 
that
\begin{equation}\label{phi+-prime-form}
\omega_+(x) = u_0 \frac{\prod_{i=1}^g E(x, q_i(\bm{u}) )^2 \prod_{k=1}^{N + g-1} E(x, u_k)}
{\prod_{j=1}^{g-d} E(x,\check{q}_k)^2} (\sigma(x))^{2}.
\end{equation}
Here $u_0 \in \bC^\ast$ is a scaling factor,
and $E(p,q)$ is the prime form on $\tilde{X} \times \tilde{X}$,
where $\tilde{X}$ is a fundamental domain of $X$ obtained by cutting along a basis of canonical cycles. The definition of $\sigma(x)$, 
which is a multi-valued $(g/2)$-differential,
in terms of the Theta function and prime forms can be found in 
Appendix B of \cite{DT1}.


\section{Opers with apparent singularities and their moduli}
\subsection{Basic definition and properties}
A branched projective structure subordinate to the Riemann surface $X$ 
is a collection $\{(U_\alpha, w_\alpha) \}$ 
where $\{ U_\alpha \}$ is a covering of $X$ 
and $\{ w_\alpha \}$ a collection of local holomorphic maps 
from $X$ to $\bP^1$ whose values are related by Möbius transform.
The ramification points of $w_\alpha$ are called apparent singularities.
Two branched projective structures are equivalent if their union is also 
a branched projective structures.
The analytic continuation of any local map $w_\alpha$ is also called 
a developing map; it defines a monodromy representation 
$\pi_1 \rightarrow PSL_2(\bC)$.
It is clear that the equivalence class of a branched 
projective structure is determined by such a developing map.

An equivalent description of branched projective structures is 
that of a $PSL_2(\bC)$-oper with apparent singularities, 
which we from now on will call \textit{oper} for short.
An oper is a collection $\cD$ of compatible 
local Schrödinger differential operators $\cD = \{ (U_\alpha, D_\alpha) \}$ where 
\begin{align*}
    &D_\alpha = \lambda \partial_{z_\alpha}^2 + q_\alpha(z_\alpha),
    &\lambda \in \bC^\ast.
\end{align*}
Here, ``compatible'' means that solutions to $D_\alpha$ are 
sections of a line bundle $N$ of degree $1-g$ (such as $K^{-1/2}_X$), 
namely if $f_\alpha$ is a solution to $D_\alpha$ then 
$N_{\beta\alpha} f_\alpha$ is a solution to $D_\beta$
\cite{Iwa91, Iwa92}. 
The two equivalent definitions of branched projective 
structures are related 
by taking the ratio $f_{1,\alpha}/f_{2, \alpha}$ of two linearly 
independent solutions to $D_\alpha$ to be $w_\alpha$ 
and by noting that  
\begin{equation}\label{potential-Schwarzian}
    q_\alpha(z_\alpha) = \frac{\lambda^2}{2} \left\{ w_\alpha(z_\alpha), z_\alpha 
    \right\}
    = \frac{\lambda^2}{2} 
    \left\{ \frac{f_{1,\alpha}(z_\alpha)}{f_{2,\alpha}(z_\alpha)}, z_\alpha \right\} 
\end{equation}
where $\left\{ g(z), z \right\} \coloneqq \frac{g'''}{g'} - \frac{3}{2} \left( \frac{g''}{g'} \right)^2$ is the Schwarzian derivative of a function $g(z)$.
It follows from the transformation rule of a Schwarzian derivative upon a change 
of coordinates $z_\alpha \rightarrow z_\beta(z_\alpha)$
that the potentials of the local Schrödinger equations 
transform as 
\begin{align}\label{potential-transform}
	q_\beta(z_\beta) (z_\beta' )^2 
 = q_\alpha(z_\alpha) - \frac{\lambda^2}{2} \left\{z_\beta, z_\alpha \right\}.
\end{align}
We see that these potentials transform almost as quadratic differentials 
plus a correction term scaled with $\lambda^2$.
Note that this correction term vanishes if $z_\beta(z_\alpha)$ is a 
Möbius transform.
Consequently, w.r.t. a coordinate atlas where the coordinate changes 
are all Möbius transform, the potentials $q_\alpha$ indeed define a quadratic 
differential (Example \ref{ex-opers-wo-as} below gives examples of 
such atlases).

One can show by direct computation from \eqref{potential-Schwarzian} 
that the singularities of $q_\alpha(z_\alpha)$, i.e. the ramification points of 
$w_\alpha(z_\alpha)$, have particular Laurent tail expansion, 
which ensure that even though solutions to $D_\alpha$ are singular and might 
have non-trivial monodromy around these singularities, 
their ratios are holomorphic there.
Consequently, the developing monodromy in $PSL_2(\bC)$ do not ``see'' these singularities, 
hence the name ``apparent singularities''.
The order of an apparent singularity is the ramification order of $w_\alpha$.
In this paper, we will restrict ourselves to opers with \textit{simple} 
apparent singularities, 
namely they are of order $1$.
It follows from \eqref{potential-Schwarzian} and the Taylor expansion 
$w_\alpha(z_\alpha) = \sum_{k\geq 1} w_k z_\alpha^k$ 
at a simple apparent singularity that the Laurent expansion of 
$q_\alpha(z_\alpha)$ takes the form 
\begin{align}\label{as-condition}
    &\frac{1}{\lambda^2} q_\alpha(z_\alpha) = 
    - \frac{3}{4z_\alpha^2} + \frac{\nu_\alpha}{z_\alpha} + q_{\alpha,0} + \mathcal{O}(z_\alpha),
    &\nu_\alpha^2 + q_{\alpha,0} = 0.
\end{align}
We call $\nu_\alpha$ the \textit{residue parameter} of the apparent singularity 
$z_\alpha = 0$ of the oper $\cD$ w.r.t. local coordinate $z_\alpha$.

\begin{example}\label{ex-opers-wo-as}
The space of opers without apparent singularities on $X$ is an affine 
space modeled over $H^0(X, K_X^2) \simeq \bC^{3g-3}$.
Let $\{U_\alpha, z_\alpha\}$ be a collection of coordinates charts 
induced by the universal covering of $X$, 
i.e. the upper-half plane, via the uniformisation theorem.
In these coordinates, the ``uniformising'' oper takes the local 
form $D_\alpha = \lambda \partial_{z_\alpha}^2$.
If a quadratic differential has local expression of the form 
$q_\alpha(z_\alpha) dz_\alpha^2$ then the collection of differential operators
$\lambda \partial_{z_\alpha}^2 + q_\alpha(z_\alpha)$ define a meromorphic oper 
without apparent singularities. 
Note that the ratios $w_\alpha$ of linearly independent solutions 
of any oper without apparent singularity 
define coordinate atlas $\{ (U_\alpha, w_\alpha) \}$ 
with coordinates changes being Möbius transforms. 
In such an atlas, an oper with simple apparent singularity 
is equivalent to a meromorphic quadratic differential 
with poles of order $2$ satisfying \eqref{as-condition}.
\end{example}

\subsection{Moduli of opers}
\paragraph{Transformation rules of residue parameters}
Since the Schwarzian derivative of local coordinates w.r.t. each other 
is regular, it follows from \eqref{potential-transform} 
that upon a change of coordinates $z_\beta \rightarrow z_\alpha(z_\beta)$,
the transform $\nu_{\beta} \rightarrow \nu_\alpha$ 
of the residue parameters associated to a simple apparent singularity $u_n$
is the same as the transform of residue parameters of a quadratic differential 
with a double pole at $u_n$.
By a direct computation one can show that  
\begin{align}\label{residue-transform}
    \nu_\alpha = \frac{\nu_\beta}{z_\alpha'(u)}  
    + \frac{3}{4} \frac{z_\alpha''(u)}{(z_\alpha'(u))^2}
\end{align}
where we have evaluated the derivatives of $z_\alpha(z_\beta)$ 
at $z_\beta(u)$.
(The invariance of the leading Laurent series coefficient $-3/4$ 
is on the other hand characteristic of quadratic differentials with double poles.)
It will turn out later to be more convenient to use the 
$\lambda$-scaled residue parameter 
\begin{equation}\label{scaled-residue}
    \nu_{\lambda, n} \coloneqq \lambda \nu_n.
\end{equation}
In terms of these parameters, the transformation rule is that of 
fiber coordinates of $T^\ast X$ 
plus a correction term linear in $\lambda$, namely
\begin{align}\label{residue-transform-scaled}
    \nu_{\lambda, \alpha} = \frac{\nu_{\lambda, \beta}}{z_\alpha'(u)}  
    + \frac{3 \lambda}{4} \frac{z_\alpha''(u)}{(z_\alpha'(u))^2}.
\end{align}

\paragraph{Opers in terms of apparent singularities and residue 
parameters}
On the other hand, one can characterise opers with $3g-3$
apparent singularities $\bu = \sum_{n=1}^{3g-3} u_n$
in terms of the positions and residue parameters of these 
singularities, 
provided that these there is no quadratic differential vanishing at $\bu$.
The idea in the proof of the following proposition, 
which uses ``building blocks'' of quadratic differentials 
to control the residue parameters,
was first used by Iwasaki in \cite{Iwa91}.

\begin{proposition}\label{prop-characterise-opers}
\cite{D-thesis, Iwa91} 
Let $\sum_{n=1}^{3g-3} u_n$ be a reduced divisor on $X$ 
such that there is no quadratic differential vanishing at $\bu$.
Then there exists coordinate neighborhoods $(U_n, z_n)$ 
of each $u_n$, $n \in \{1, ..., 3g-3\}$,
and an injective map of sets
\begin{align*}
\left( U_1 \times ... \times U_{3g-3} \right) \times (\bC)^{3g-3} 
&\longrightarrow \{ \text{opers with $3g-3$ simple apparent singularities} \} / \sim 
\\
(\Vec{u'}, \Vec{\nu}) &\longmapsto \cD(\Vec{u'}, \Vec{\nu})
\end{align*}
such that the oper $\cD(\Vec{u'}, \Vec{\nu})$ associated to 
$(\Vec{u'}, \Vec{\nu}) = (u_1', ..., u_{3g-3}', \nu_1, ..., \nu_{3g-3})$
has a simple apparent singularity at each $u_n'$
and associated residue parameters $\nu_n$
w.r.t. the local coordinates $z_n$.
\end{proposition}
\begin{proof}
The idea is that due to our hypothesis on $\sum_{n=1}^{3g-3} u_n$, 
we can find local neighborhoods $U_n$ of $u_n$ such that  
points in $U_1 \times ... \times U_{3g-3}$ also define divisors 
$u_1' + \dots + u_{3g-3}'$ at which no quadratic differential vanish.
 Choose an oper without apparent singularity and equip $U_n$ 
with local coordinate $z_n$ induced by this oper 
(cf. Example \ref{ex-opers-wo-as}).
We now describe $\cD(\Vec{u'}, \Vec{\nu})$ 
by a quadratic differentials with double poles $u_n' \in U_n$ 
and Laurent expansion in $z_n$ satisfying \eqref{as-condition}.
To this end, we can use quadratic differentials with Laurent tails
$$q^{(2)}_{u_n'} = \frac{1}{(z_n - z_n(u_n'))^2} + \mathcal{O}(z_n - z_n(u_n')),
\qquad  q^{(1)}_{u_n'} = \frac{1}{z_n - z_n(u_n')} + \mathcal{O}(z_n - z_n(u_n')) $$
around $u_n'$ and vanish at $u_k'$ for $k \neq n$.
These quadratic differentials are actually unique.
The quadratic differential 
$$ -\frac{3}{4} \sum_{r=1}^{3g-3} q^{(2)}_{u_n'} 
+ \sum_{r=1}^{3g-3} \nu_r q^{(1)}_{u_n'} + q^{(0)}_{(\Vec{u'}, \Vec{\nu})} $$
where $q^{(0)}_{(\Vec{u'}, \Vec{\nu})}$ 
is the unique holomorphic quadratic differential 
that solves the nondegenerate linear system 
\begin{equation}\label{linear-system}
q^{(0)}_{(\Vec{u'}, \Vec{\nu})}(z_n)\mid_{z_n(u_n')} + \nu_n^2 = 0,
\qquad n = 1, \dots, 3g-3,    
\end{equation}
then defines $\cD_{(\Vec{u'}, \Vec{\nu})}$. 
The injectivity of the assignment 
$(\Vec{u'}, \Vec{\nu}) \mapsto \cD_{(\Vec{u'}, \Vec{\nu})}$ is obvious.
\end{proof}

\paragraph{$Q$-generic and $Q$-special divisors}
The hypothesis on $\bu$ in Proposition \ref{prop-characterise-opers} is a 
case of the notion of $Q$-generic divisors \cite{D-thesis, D}. 
We say that an effective divisor 
$\bu = \sum_{n=1}^m u_n$ on $X$ is $Q$-generic if the space 
$$ Q_\bu = \{ q \in H^0(X, K_X^2) \mid \bmu < \mathrm{div}(q) \} \cup \{ 0 \}$$
of quadratic differentials vanishing at $\bu$, with multiplicity counted, 
is of expected, i.e. minimal, dimension
$$ \max\{0, 3g-3 - \deg(\bu)\} . $$
We say that $\bu$ is $Q$-special otherwise.
For $\deg(\bu) \geq 3g-3$, this means $\bu$ is $Q$-special if and only if 
there is a quadratic differential vanishing at $\bu$. 
Note that the space of solutions to \eqref{linear-system}, if they exist, 
is an affine space modeled over $Q_\bu$.

It follows from the constructive proof in Propostion \ref{prop-characterise-opers} 
that whenever $\bu$ is reduced and $\dim(Q_\bu) > 0$,
there exist families of opers with 
simple apparent singularities at $\bu$ and the same residue parameters. 
Such families are affine spaces modeled over $Q_\bu$.
For example, such are generic cases for $\deg(\bu) < 3g-3$.
The following proposition follows by combinning this observation 
with the proof of Proposition \ref{prop-characterise-opers}.

\begin{proposition}\label{prop-characterise-opers-2}
\cite{D-thesis, Iwa91} 
For $m \leq 3g-3$,
let $\sum_{n=1}^{m} u_n$ be a reduced $Q$-generic divisor on $X$.
Then there exists local coordinates $z_n$ around $u_n$ 
and a non-canonical injective map of sets
\begin{align*}
(\bC)^{3g-3} \times Q_\bu
&\longrightarrow \{ \text{opers with $3g-3$ simple apparent singularities at }\bu \} / \sim 
\\
(\Vec{\nu}, \Delta q) &\longmapsto \cD(\Vec{\nu}, \Delta q)
\end{align*}
such that the oper $\cD(\Vec{\nu}, \Delta q)$ 
has associated residue parameters $\nu_n$ 
w.r.t. the local coordinates $z_n$.
\end{proposition}

\begin{remark}
\begin{enumerate}
\item Proposition \ref{prop-characterise-opers} 
does not hold for a $Q$-special divisor 
$\bmu' = \sum_{n=1}^{3g-3} u_n'$: 
in fact, for most vectors $\Vec{\nu} \in \bC^{3g-3}$, 
there is no oper with apparent singularities at $\bmu'$
with residue parameters defined by $\Vec{\nu}$ since the 
system \eqref{linear-system} has no solution. 
On the other hand, there exist families of opers 
with apparent singularities at $\bmu'$ with the same 
residue parameters. These families are affine spaces modeled over 
the vector space $Q_{\bmu'}$.
\item It can be checked using \eqref{residue-transform} that upon a Möbius 
transform of local coordinates, the residue parameters can be put to $0$.
This means that fixing an oper without apparent singularity 
in the proof of Proposition \ref{prop-characterise-opers}
does not yet fix the choice of the oper $\cD(\Vec{u'}, \Vec{0})$ 
for each fixed $\Vec{u'}$.
Rather, $\cD(\Vec{u'}, \Vec{0})$ 
depends on the choices of the coordinates $z_n$ themselves.
\item \vspace{-8pt} In contrast with the Riemann-Hilbert correspondence between 
isomorphism classes of flat connections on vector bundles 
and conjugacy classes of monodromy representations in e.g. $SL_2(\bC)$,
in general there exist non-equivalent 
opers with apparent singularities with the same monodromy representation 
in $PSL_2(\bC)$.
\end{enumerate}
\end{remark}

\paragraph{Moduli of opers with apparent singularities}
Proposition \ref{prop-characterise-opers} 
and the transformation rule \eqref{residue-transform-scaled}
motivate us to consider the affine bundle 
$\mM'_{\mathrm{op}} \rightarrow X$ modeled over $T^\ast X$ with transition 
function defined as in \eqref{residue-transform-scaled}.
Let 
\begin{align}
    \mM^m_{\mathrm{op}} \coloneqq \left( \mM'_{\mathrm{op}} \right)^m / S_m.
\end{align}
One can regard $\mM^m_{\mathrm{op}}$ as 
an affine analogue of the symmetric product $\left( T^\ast X \right)^{[m]}$, 
which is recovered at the limit $\lambda \rightarrow 0$.

For $m=3g-3$,
it follows from Proposition \ref{prop-characterise-opers} that,
away from the diagonals 
\begin{equation}\label{diagonal}
    \{ [(u_1, \nu_1), \dots, (u_{3g-3}, \nu_{3g-3})] \mid u_n = u_k 
\text{ for some } n \neq k \} \subset \Mop     
\end{equation}
and the loci defined by $Q$-special divisors 
\begin{equation}\label{loci-wobbly}
    \{ [(u_1, \nu_1), \dots, (u_{3g-3}, \nu_{3g-3})] \mid \sum_{n=1}^{3g-3} u_n 
\text{ is $Q$-special} \} \subset \Mop, 
\end{equation}
we have $\Mop$ as the moduli space of opers with $3g-3$ 
simple apparent singularities which do not form $Q$-special divisors. 
The locus defined in \eqref{loci-wobbly} signifies the presence of 
wobbly bundles, which are stable bundles that admit nonzero nilpotent Higgs fields;
we refer to section \ref{sect-wobbly} for a more elaborate discussion.
For $2g-2 < m < 3g-3$, it follows from Proposition \ref{prop-characterise-opers-2}
that away from the loci \eqref{diagonal} and \eqref{loci-wobbly}, 
up to a choice of a quadratic differential,
$\Mop$ is the moduli space of opers with $m$ simple apparent singularities 
which do not form $Q$-special divisors.

\paragraph{Poisson structure} 
There exists a natural closed $2$-form on $\Mop$ 
defined away from the diagonals
which locally take the form  
\begin{equation}\label{local-form}
    \sum_{n=1}^m dz_n \wedge d\nu_n    
\end{equation}
in the coordinates defining $\Mop$. 
It is easy to see from \eqref{residue-transform} that the local $2$-form
\eqref{local-form} is invariant upon a change of coordinates 
and hence extends to a global closed 2-form on $\Mop$ 
away from the diagonals
(this observation was first made by Iwasaki \cite{Iwa91}). 

Note that in the non-compact Riemann surface setting, 
Iwasaki \cite{Iwa92} showed that \eqref{local-form} defines a symplectic form 
which coincides with the 
pull-back of the canonical symplectic form in the character variety 
$$ \mathrm{Hom}(\pi_1, PSL_2(\bC)) / \sim $$
along the monodromy map. 
We expect an analogous result holds in the compact Riemann surface setting.

\section{Separation of Variables as a Poisson map}
\subsection{From triples $(E, L, \nabla_\lambda)$ to opers}
\label{sect-define-SOV}
\paragraph{Two methods of inducing opers}
Given data $(L \hookrightarrow F, \nabla_\lambda)$, 
we discuss in the following two methods to induce an oper with apparent singularities.
The opers defined by these two methods are equivalent.

The first way is geometric: by projectivising the flat bundle $F^{\nabla_\lambda}$ 
defined by $(F, {\nabla_\lambda})$, one gets a flat $PSL_2(\bC)$-bundle 
$\bP(F^{\nabla_\lambda})$ with $\bP^1$-fibers. 
projectivising further the sub-line bundle of $F^{\nabla_\lambda}$ induced by $L$ yields 
a section of $\bP(F^{\nabla_\lambda})$. This section provides local maps to $\bP^1$
whose values are related by constant Möbius transform. 


The second method is analytic and gives explicit expression of the induced opers 
in terms of local Schrödinger equations w.r.t. local coordinates. 
Suppose in local frames adapted to $L$ and in local coordinate $z$ we have 
\begin{equation}\label{local-form-nabla}
    \nabla_\lambda(z) = \lambda \partial_z + 
    \begin{pmatrix}
        a(z) & b(z) \\ c(z) & -a(z)
    \end{pmatrix}.
\end{equation}
Consider the local function $g_\lambda(z) = a(z) - \frac{\lambda}{2} \frac{c'(z)}{c(z)}$ 
and the local differential operator $\lambda^2 \partial_z^2 + q(z)$ where
\begin{align}\label{induce-potential}
    q(z) &= b(z) c(z) + (g_\lambda(z))^2 + \lambda g_\lambda'(z) \nonumber \\
    &= b(z) c(z) + \left( a(z) - \frac{\lambda}{2} \frac{\partial_z c(z)}{c(z)} \right)^2 
    + \lambda \hspace{1pt} \partial_z \left( a(z) - \frac{\lambda}{2} \frac{\partial_z c(z)}{c(z)} \right).
\end{align}
It is straightforward to check that $q(z)$ is invariant upon a change of 
local frames adapted to $L$ 
w.r.t. the local $1$-forms $a(z)$, $b(z)$ and $c(z)$ are defined. 
One can also check that, upon a change of coordinates and transformations of $1$-forms
\begin{align*}
    &a(z_\alpha) \rightarrow a(z_\alpha(z_\beta)) z_\alpha'(z_\beta),
    & &b(z_\alpha) \rightarrow b(z_\alpha(z_\beta)) z_\alpha'(z_\beta),
    & &c(z_\alpha) \rightarrow c(z_\alpha(z_\beta)) z_\alpha'(z_\beta),
\end{align*}
we have $q(z)$ transform as in \eqref{potential-transform}. 
We conclude that the collection of local differential operators 
$\lambda^2 \partial_z^2 + q(z)$ defines an oper on $X$.
It is straightforward to check that the apparent singularities of this oper 
are defined by the zeroes of $c(L \hookrightarrow E, \nabla_\lambda)$ 
with mulitiplicities counted.

\begin{proposition}
Consider an irreducible $\lambda$-holomorphic connection $(E, \nabla_\lambda)$ 
together with a subbundle $L\hookrightarrow E$.
Then the oper $\cD = \{ \lambda^2 \partial_z^2 + q(z) \}$ 
with $q(z)$ defined in \eqref{induce-potential} 
is equivalent to the oper defined by projectivising the data 
$(L \hookrightarrow E, \nabla_\lambda)$. 
In particular, for $\lambda = 1$,  
the $PSL_2(\bC)$-monodromy representation of $\cD$ is the projection
of the $SL_2(\bC)$-monodromy representation of $(E, \nabla_\lambda)$
up to conjugation.
\end{proposition}
\begin{proof}
It suffices to prove equivalence in a neighborhood $U \subset X$ that contains 
no apparent singularity. Let $w$ be the ratio of local solutions to 
$\lambda^2 \partial_z^2 + q(z)$. Refining $U$ if necessary, we can use $w$ 
as a local coordinate on $U$ w.r.t. which the oper can be represented by the 
differential operator $\partial_w^2$. On the other hand, upon choosing a square-root 
$\sqrt{c(w)}$, observe that the effect of the holomorphic gauge transformation 
$G_\lambda = \left( \begin{smallmatrix} c(w)^{-1/2} & 0 \\ 0 & c(w)^{1/2} \end{smallmatrix} \right)
\left( \begin{smallmatrix} 1 & g_\lambda(w) \\ 0 & 1 \end{smallmatrix} \right)$ 
is that
\begin{align}\label{gauge-transformation-effect}
   \nabla_\lambda(w) = \partial_w + 
\begin{pmatrix} a(w) & b(w) \\ c(w) & -a(w) \end{pmatrix} 
\longmapsto 
\partial_w + \begin{pmatrix} 0 & -q(w) \\ 1 & 0 \end{pmatrix}
\end{align}
where $q(w)$ is defined by the same formula as in \eqref{induce-potential}.
This implies that $q(w) = 0$. 
Then $\left( \begin{smallmatrix} 0 \\ 1 \end{smallmatrix} \right)$
and $\left( \begin{smallmatrix} -1 \\ w \end{smallmatrix} \right)$
are local parallel sections. In the local frame defined by these two sections, 
a generator of $L^{\nabla_\lambda}$ takes the form
\begin{equation*}
    \begin{pmatrix} 0 & -1 \\ 1 & w \end{pmatrix}^{-1} 
    G^{-1} \begin{pmatrix} 1 \\ 0 \end{pmatrix}
    = c(w)^{1/2} \begin{pmatrix} -w \\ 1 \end{pmatrix}.
\end{equation*}
The local function defined by $\bP(L^{\nabla_\lambda})$ is hence $-w$.  
\end{proof}

\begin{remark}
\begin{enumerate}
    \item It is well-known that gauge transformations of the form 
$G_\lambda$ put holomorphic connections to the canonical form 
\eqref{gauge-transformation-effect}
(In the $g = 0$ setting this defines a map preserving the canonical 
symplectic structures related to the Schlesinger and Garnier systems \cite{DM}).
In fact it is this gauge transformation 
that motivates our definition of induced opers in \eqref{induce-potential}.
Note that while we might need to pick a system of branch cuts to define 
the gauge transformation around the zeroes of $c(E, L, \nabla_\lambda)$,
our definition of oper in \eqref{induce-potential}, 
does not depend on such choices.
    \item Twisting the data $(L \hookrightarrow E, \nabla_\lambda)$
    by a square-root of $\mathcal{O}_X$ induce opers up to equivalence.
\end{enumerate}
\end{remark}

\paragraph{Induced residue parameters}
At a simple zero $u_n$ of $c(E, L, \nabla_\lambda)$ 
the potential $q(z)$ has the Laurent expansion
\begin{equation}
    \frac{-3 \lambda^2}{4} \frac{1}{(z-z(u_n))^2} 
    + \lambda \left[ a(u_n) - \lambda \frac{c''(u_n)}{4c'(u_n)} \right] \frac{1}{z - z(u_n)}
    - \left[ a(u_n) - \lambda \frac{c''(u_n)}{4c'(u_n)} \right]^2 + \mathcal{O}(z - z(u_n)).
\end{equation}
In the conventions of \eqref{as-condition} and \eqref{scaled-residue}, 
we have
\begin{align}\label{induced-parameters}
   \nu_{\lambda, n} \coloneqq \lambda \nu_n  
    = a(u_n) - \lambda \frac{c''(u_n)}{4c'(u_n)}.
\end{align}
Note that a scaling of the embedding $L \hookrightarrow E$ 
does not change the residue parameter. 
It follows from Proposition \ref{prop-characterise-opers} that 
from the data of a generic point 
$\bmx = (L \hookrightarrow E) \in \M$ together with a $\lambda$-connection $\nabla_\lambda$ 
on $E$ one can induce a unique oper 
which is invariant along the $\bC^\ast$-orbit.

\paragraph{Inverse construction} The construction of opers from triples 
$(E, L, \nabla_\lambda)$ can also be inverted. 

\begin{proposition}\label{prop-inverse-SOVdR}
For an even positive integer $m$, 
let $\cD = 
\{ D_\alpha = \lambda^2 \partial_{z_\alpha}^2 + q_\alpha(z_\alpha) \}$ 
be an oper with $m$ simple apparent singularities. 
Then there exists a triple $(E, L, \nabla_\lambda)$ that induces 
an oper equivalent to $\cD$ by the operation described above. 
Furthermore, for $\lambda = 1$,
if $(E, \nabla_\lambda)$ is irreducible then 
such a triple is unique up to tensoring with a square-root of $\mathcal{O}_X$.
\end{proposition}
\begin{proof} (Sketch)
See Proposition 5.6 in \cite{D-thesis} for a detailed constructive proof 
for the case $\lambda = 1$.
The case for $\lambda \in \bC^\ast$ can be generalised in a straight-forward manner.
The idea is to apply Hecke transformation of bundles to a split bundle 
at the apparent singularities of $\cD$. The choices of the flags at which we do 
Hecke transformation can be chosen such that the residue parameters 
and the components $a(z)$, $c(z)$ 
of $\nabla_\lambda$ satisfy \eqref{induced-parameters}.
This defines the bundle $E$ and the embedding of $L$ into $E$.
The component $b(z)$ of $\nabla_\lambda$ can be defined by inverting 
\eqref{induce-potential}.
The uniqueness statement follows from the Riemann-Hilbert correspondence 
for $G = SL_2(\bC)$ and the fact that the 
monodromy representation of $\cD$ has $2^{2g}$ lifts to $SL_2(\bC)$.
\end{proof}

We note that the triple $(E, L, \nabla_\lambda)$ constructed from 
$\cD$ is such that $\nabla_\lambda$ is not $L$-invariant, but 
a priori is not necessarily irreducible. However, for a fixed reduced divisor 
$\bu = \sum_{n=1}^m u_n$ with $m \leq 3g-3$, then a generic 
\footnote{Recall Proposition \ref{prop-characterise-opers-2}: 
opers with apparent singularities at $\bu$ can be classified 
by their residue parameters and holomorphic quadratic differentials vanishing 
at $\bu$. Genericity then can be defined w.r.t. this classification.}
choice of opers with apparent singularities 
at $\bu$ will be induced by $(E, L, \nabla_\lambda)$ with irreducible $(E, \nabla_\lambda)$.

\subsection{Construction of the Separation of Variables maps}
We now briefly review construction of Baker-Akhiezer divisors 
from triples $(E, L, \phi)$ and the map $\SOV$ 
for Higgs bundles from \cite{DT1}.
The purpose is to show that the apparent singularities and residue parameters 
of opers constructed from triples 
$(E, L, \nabla_\lambda)$ are analogues of Baker-Akhiezer divisors.
We then construct $\SOVdR$, the analogue 
of $\SOV$ for holomorphic connections.

\paragraph{Baker-Akhiezer divisors of triples $(E, L, \phi)$} 
Let $\left(E, \phi \right)$ be a rank-2 $SL_2(\bC)$-Higgs bundle
and suppose the quadratic differential $q = \det(\phi)$ has only simple zeroes. 
The spectral curve $S \overset{\pi}{\rightarrow} X$
is a subset of $T^\ast X$ defined by taking the square-roots of $q$, 
i.e. solving for the eigenvalues of $\phi$.
There is the involution $\sigma: S \rightarrow S$ that 
acts on the fibers of $\pi$ exchanging the two square-roots.
By solving for the eigenvectors of $\phi$, one can define a line bundle 
$\cL_{(E, \phi)}$ on $S$ called the eigen-line bundle.
One can recover the Higgs bundle by taking the direct image 
of $\cL_{(E, \phi)} \otimes \pi^\ast(K)$ along $\pi$.
Hence up to isomorphisms, Higgs bundles $(E,\phi)$ with $\det(\phi) = q$ 
are in 1-1 correspondence with line bundles on $S$ 
(of appropriate degrees) -- this is known as 
Hitchin's spectral correspondence for smooth spectral curves \cite{Hit}.

Given a line bundle $L$ with an injection $L \rightarrow E$ 
(which possibly has zeroes), 
we can consider on $S$ the composition
\begin{align}
	\pi^\ast\left( L \right) \longrightarrow \pi^\ast\left(E \right) \longrightarrow  \cL_{(E, \phi)}^{-1} \pi^\ast\left( \det(E) \right)
	\label{BA-div-def-composition}
\end{align}
where the second map is the quotient of the embedding 
$\cL_{(E, \phi)} \hookrightarrow \pi^\ast(E)$.
The zero divisor of this composition consists of 
the pull-back of the zero divisor of $L \rightarrow E$ 
and points where $\pi^\ast(L)$ coincide with $\cL_{(E, \phi)}$ as subbundles of $\pi^\ast(E)$.

\begin{definition}
Let $(E, \phi)$ be
a Higgs bundle on $X$ with non-degenerate spectral curve $S \overset{\pi}{\rightarrow} X$
and $L$ a line bundle with an injection $L \rightarrow E$.
The \textit{Baker-Akhiezer (BA) divisor} 
associated to the data $(L \rightarrow E, \phi)$ is the 
 involution of the zero divisor of the composition $\pi^\ast\left( L \right) \rightarrow \pi^\ast\left(E \right) \rightarrow  \cL_{(E, \phi)}^{-1} \pi^\ast\left( \det(E) \right)$.
	\label{def-BA divisor-formal}
\end{definition}

BA divisors can be characterised concretely if $L$ is a subbundle of $E$, 
namely the injection $L \rightarrow E$ is nowhere vanishing. 
Recall the section $c(E, L, \phi)$ of the line bundle $KL^{-2} \det(E)$ 
from \eqref{lower-left-Higgs}, 
and suppose $\phi$ takes the form
$ \left(\begin{smallmatrix} a_\alpha & b_\alpha \\ c_\alpha 
& - a_\alpha \end{smallmatrix} \right)$
in local frames adapted to $L$.
The preimages along $S \overset{\pi}{\rightarrow} X$ 
of a zero $u_n$ of $c_\alpha \equiv c(E, L, \phi)$
consist of two points (counted with multiplicity), 
labeled by the eigenvalues $\pm a_\alpha(u_n)$ 
of $\phi$ at $u_n$.
Then the BA divisor of $\left( E, L,  \phi \right)$ 
is $\tbu = \sum_n \tilde{u}_n$ where $\tilde{u}_n$ is the point labeled by 
\begin{equation}\label{residue-Higgs}
    v_n = a_\alpha(u_n).    
\end{equation}
We see that the apparent singularities 
and the projection to $X$ of BA divisors are 
respectively the zeroes of $c(E, L, \nabla_\lambda)$ and 
$c(E, L, \phi)$.
The residue parameter $\nu_n$ in \eqref{scaled-residue}
can be regarded as a deformation of $v_n$ by a term linear in $\lambda$.

We refer to \cite{D-thesis, DT1} for the proof of the following proposition.
\begin{proposition} \label{prop-properties-BA}
Let $\tbu$ be the BA divisor of $\left( L \rightarrow E, \phi \right)$ on a nondegenerate spectral curve $S \overset{\pi}{\rightarrow} X$,
 and $\bu = \pi(\tbu)$.
Then:
\begin{enumerate}
     \item $\tbu$ has a summand of the form $\pi^{\ast}(\bu')$ 
     if and only if $L \rightarrow E$ vanishes at $\bu'$, 
     counted with multiplicity. 
     In particular, $\tbu$ contains no summand 
     equal to the pull-back of a divisor on $X$ 
     if and only if $L$ is a subbundle of $E$. 
     \item The projection $\bu$ satisfies
     $\mathcal{O}_X(\bu) \simeq KL^{-2}\det(E)$.
     If $L$ in particular is a subbundle of $E$, 
     then $\bu$ coincides with the zero divisor of $c_{i}(\phi)$.
     \item The eigen-line bundle $\cL_{(E, \phi)}$ of $\left(E,\phi \right)$ is isomorphic to $\pi^\ast\left(L K^{-1} \right) \otimes \mathcal{O}_S(\tbu)$. 
     In other words, $(E, \phi)$ is isomorphic to the direct image of 
     $\pi^\ast\left(L \right) \otimes \mathcal{O}_S(\tbu)$.
\end{enumerate}
\end{proposition}

Let us define an isomorphism class $[ L \rightarrow E, \phi ]$
of the input data of BA divisors
by saying that two representative data 
are isomorphic if there are isomorphisms of the underlying bundles and line bundles
that commute with the injections and Higgs fields
    \footnote{Since scalings are isomorphisms of line bundles,
    scaling the injections from line bundles
    to rank-2 bundles 
    will define the same isomorphism class $[ L \rightarrow E, \phi ]$.}.
Clearly BA divisors defined by isomorphic data coincide.
The following theorem
summarizes the invertible properties of 
the construction of BA divisors.
We refer to the discussion following the proof of theorem 8.1 in Hitchin's original work \cite{Hit} for an abstract proof 
and \cite{D-thesis, DT1} for a more explicit proof in terms of 
Abelian differentials.

\begin{theorem}\label{thm-BA}
Let $q$ be a quadratic differential whose zeroes are all simple 
and $S \overset{\pi}{\rightarrow} X$ its corresponding spectral curve.
Then the construction of BA divisors and remembering the line bundles define a bijection
\begin{align*}
	\left\{ [ L \rightarrow E, \phi ] \ | 
		\begin{array}{l}
			\det(\phi) = q \\		 
	\end{array} \right\}
	\longleftrightarrow 
	\left\{ ( [L], \tbu ) \ \bigg| \begin{array}{l}
			\tbu \text{ effective on } S,  \\
			K L^{-2}\det(E) \simeq \mathcal{O}_X(\pi(\tbu) )
	\end{array} \right\}.
\end{align*}
In particular, in the inverse direction, the injection
\begin{align}
L \longrightarrow E \simeq L \otimes \pi_\ast(\mathcal{O}_S(\tbu)),
\end{align}
is obtained by taking the direct image of the canonical section of $\mathcal{O}_S(\tbu)$.
\end{theorem}

\paragraph{The map $\SOV$}
For $0 < -2d \leq g-1$,
consider a cotangent vector on $\N$ at $(E, L)$ 
which is the pull-back of some Higgs field $\phi$ on $E$ via 
$i^\ast: T^\ast_E \Nstable \rightarrow T^\ast_{(E,L)}\N$ 
where the spectral curve $S_\phi$ associated to $(E, \phi)$ is smooth.
We then can assign a BA divisor on $S_\phi$ which defines a configuration of 
$m$ points in $T^\ast X$.
If $i^\ast(\phi) = i^\ast(\phi')$ for another Higgs field $\phi'$ on $E$,
then $\phi - \phi'$ is a nilpotent Higgs field.
The BA divisor on $S_{\phi'}$ then defines the same configuration of points in 
$T^\ast X$. 
Restricting to the locus of $T^\ast \N$ where the zeroes of the 
corresponding $c(E, L, \phi)$ are all simple, 
we can define a rational map
\begin{align*}
&\SOV: T^\ast\N \dashrightarrow ( T^\ast X)^{[m]},
&m = 2g-2 - 2d.
\end{align*}

The following is the main result in \cite{DT1}.
\begin{theorem} (Theorem \ref{thm-SOV-Higgs})
    $\SOV$ is a dominant Poisson map w.r.t. canonical symplectic structures 
    whose generic fibers are $2^{2g}:1$.
\end{theorem}
On one hand,
the map $\SOV$ is the high genus generalisation of the Separation of Variables 
technique applied to the classical Gaudin model, which is 
a variant of the Hitchin moduli space for $g=0$ \cite{Sklyanin, T18}.
On the other hand, it can be regarded as the classical limit of 
Drinfeld's approach to the geometric Langlands correspondence 
adapted to $G= SL_2(\bC)$ setting \cite{DT1, Drinfeld, Fre95}.

\paragraph{The map $\bm{\SOVdR}$} 
We have seen that from the data $(E, L, \nabla_\lambda)$ 
where $c = c(E, L, \nabla_\lambda)$ has simple zeroes, 
one can induce an oper with simple apparent singularities at 
$\mathrm{div}(c)$ and residue parameters \eqref{induced-parameters}.
By the same argument in defining $\SOV$ 
(cf. Proposition \ref{prop-surjectivity-failure}), 
for $0 < -2d \leq g-1$ and fixed $\lambda \in \bC^\ast$,
we can define a rational map
\begin{align*}
&\SOVdR: \MdRd \dashrightarrow \Mop,
&m = 2g-2 - 2d.
\end{align*}

In the next subsection, we will show that $\SOVdR$ is a Poisson map 
w.r.t. the Poisson structures we have defined on $\MdRd$ and $\Mop$.
To this end, it is more convenient to consider its lift 
to $H^{-1}(\lambda d)$, namely
\begin{equation}\label{SOVdR}
\begin{tikzcd}[column sep = large]
&\tMdRd \arrow[dashed]{r}{\widetilde{\SOVdR}} \supset H^{-1}(\lambda d) \arrow{d}
&\Mop \\
&\MdRd \arrow[dashed]{ur}[swap]{\SOVdR} &
\end{tikzcd}.    
\end{equation}

\subsection{Proof of Poisson property}
\label{sect-proof}
Let us state again the main theorem of this paper.
\begin{theorem} \label{thm-SOV-dR-2} (Theorem \ref{thm-SOV-dR-1})
For $0 < -2d \leq g-1$, the map 
\begin{align}
    &\SOVdR: \MdRd \dashrightarrow \Mop,
    &m = 2g-2 -2d,
\end{align}
is a dominant Poisson map.
\end{theorem}

The proof of Theorem \ref{thm-SOV-dR-2} closely mirrors that of Theorem 
\ref{thm-SOV-Higgs} in \cite{DT1}, 
except that there are additional $\lambda$-dependent 
terms in several formulas.
In the following we shall outline these formulas 
to keep track of these terms and to show that  
in the key step (Lemma \ref{lemma-crucial}), these terms 
do not pose any significant challenges.

The idea of the proof 
is to write the map $\SOVdR$ rather explicitly,
upon choosing a reference divisor $\bmr$,
in terms of the Darboux coordinates $(\bmx, \blambda, \bmk, \bkappa)$ 
of $\MdRd$.
To this end, we again express the data 
$(E, L, \nabla_\lambda)$ 
in terms of the restriction of $\nabla_\lambda$ 
to the open dense set $X_{\bq} \subset X$, 
$$ \nabla_\lambda\mid_{X_{\bq}}
= \lambda \partial 
+ \begin{pmatrix} \omega_0 & \omega_- \\ \omega_+ & -\omega_0 \end{pmatrix}, $$
(cf. \eqref{Abelian-differentials} and Proposition \ref{prop-abelian-differentials}).
Recall from section \ref{sect-abelian-diff}
that the Abelian differential $\omega_0$ can be expressed as
\begin{equation}
    \omega_0 = - \sum_{r=2}^N x_r k_r \omega^{\bq}_{p_r - p_1} 
    + \lambda \sum_{i=1}^g \omega^{\bq}_{q_i - p_1}
    - \lambda \sum_{i=1}^{g-d} \omega^{\bq}_{\check{q}_i - p_1} 
    - \frac{1}{2} \sum_{i=1}^g \kappa_i \omega_i
\end{equation}
where $\omega^{\bq}_{x-y}$ is an Abelian differential of the third kind
whose $0$-th Laurent coefficient at $q_i$ vanishes 
(cf. \eqref{normalise-vanish-q}).
The scaled residue parameter $\nu_{\lambda,n}$ 
associated to the apparent singularity $u_n$ is defined by 
\begin{align}\label{conjugate-lambda}
    &\nu_{\lambda,n} = \omega_0(u_n) - \lambda \frac{\omega_+''(u_n)}{4\omega_+'(u_n)},
    &n = 1, \dots, m,
\end{align}
where have taken derivative and evaluations of $\omega_+$
using some chosen local coordinates around $u_n$.
For this, 
we can write $\omega_0(u_n)$ even more explicitly by plugging the chain rule 
\begin{equation}\label{identity-Abel}
	\frac{\partial q_j(\bu)}{\partial u_n} 
	= \sum_{i=1}^g \frac{\partial q_j}{\partial \lambda_i} \bigg\vert_{\bm{\lambda}(\bu)} \frac{\partial \lambda_i}{\partial u_n} \bigg\vert_{u_n} 
	= -\frac{1}{2} \sum_{i=1}^g \left( dA^{-1} \mid_{\blambda} \right)_{ij} \omega_i(u_n),
\end{equation}
in equation \eqref{relations-Abelian-diff} that relates $\omega^{\bq}_{x-y}$ 
with the canonical Abelian differentials of the third kind $\omega_{x-y}$
whose $A$-cycles vanish. We then have
\begin{align}\label{omega0-u}
    \omega_0(u_n) = &
    - \sum_{r=2}^N  k_r  x_r \bigg(  \omega_{p_r - p_1}(u_n) + 2 \sum_{j=1}^g \omega_{p_r - p_1}(q_j) \frac{\partial q_j}{\partial u_n} \bigg)  
    \nonumber \\
    &+ \lambda \sum_{i=1}^g \omega^{\bq}_{q_i - p_1}
    - \lambda \sum_{i=1}^{g-d} \omega^{\bq}_{\check{q}_i - p_1} 
    - \frac{1}{2} \sum_{i=1}^g \kappa_i \omega_i.
\end{align}

Recall from \eqref{phi+-prime-form} that 
the Abelian differential $\omega_+$ can be expressed 
in terms of the prime forms on a fundamental domain of $X$
from its zero divisor $\bmu = \sum_{n=1}^m u_n$ and pole structure 
(cf. Proposition \ref{prop-abelian-differentials}).
Its evaluation at $p_r$ hence is 
\begin{equation}\label{phi+-prime-form-2}
    k_r= u_0 \frac{\prod_{i=1}^g E(p_r, q_i(\bm{u}) )^2 \prod_{k=1}^{N + g-1} E(p_r, u_k)}{\prod_{j=1}^{g-d} E(p_r,\check{q}_k)^2} (\sigma(p_r))^{2}
\end{equation}

The key step in the proof is the following lemma.
\begin{lemma}\label{lemma-crucial}
Let $F$ be a complex-valued function on an open set in $\tMdRd$ 
equipped with local Darboux coordinates $(\blambda, \bmx, \bkappa, \bmk)$ satisfying $\{k_r,F\}=0$ and $\{\lambda_\ell,F\}=0$.
Then 
\begin{align}\label{eqn-lem-cruc} \frac{\partial}{\partial u_n} F(\blambda(\bu), \bmk(\bu))
&= - \left\{ v_{\lambda,n}, F \right\}, \quad n=1,\dots,m; \nonumber \\
 u_0\frac{\partial }{\partial u_0} F(\blambda(\bu), \bmk(\bu))
&= \left\{ \nu_{\lambda,0}, F \right\}, \quad \nu_{\lambda, 0} \coloneqq \lambda d.
\end{align}
\end{lemma}
\begin{proof}
On one hand, it follows from \eqref{phi+-prime-form-2} that 
\begin{equation*}
	\frac{1}{k_r}	\frac{\partial k_r}{\partial u_n}
	= \frac{\partial \log E(p_r, u_n)}{\partial u_n}  + 2 \sum_{i=1}^g \frac{\partial \log E(p_r, q_i)}{\partial q_i} \frac{\partial q_i(u_n)}{\partial u_n}.
\end{equation*}
We then can write
\begin{align}\label{d_un-final}
\frac{\partial F}{\partial u_n}
=  - \sum_{r=1}^N \frac{\partial k_r}{\partial u_n} \frac{\partial F}{\partial k_r} - \sum_{i=1}^g \frac{\partial \lambda_i}{\partial u_n} \frac{\partial F}{\partial \lambda_i} 
= - \sum_{r=1}^N \frac{\partial k_r}{\partial u_n} \frac{\partial F}{\partial k_r} + \frac{1}{2} \sum_{i=1}^g \omega_i(u_n) \frac{\partial F}{\partial \lambda_i}
\nonumber \\
= \sum_{r=1}^N  \bigg( \frac{\partial \log E(p_r, u_n)}{\partial u_n} 
	+ 2 \sum_{i=1}^g \frac{\partial \log E(p_r, q_i)}{\partial q_i} \frac{\partial q_i}{\partial u_n} 	\bigg) k_r \frac{\partial F}{\partial k_r}
	- \frac{1}{2} \sum_{i=1}^g \omega_i(u_n) \frac{\partial F}{\partial \lambda_i}.
\end{align}
Note that this means the LHS of the first relation in \eqref{eqn-lem-cruc} 
is not $\lambda$-dependent.

On the other hand, in computing $\{\nu_{\lambda,n}, F \}$ 
we can make use of the relations $\sum_{r=1}^N x_r k_r = \lambda d$ 
and $\omega_{p_+ - p_-}(x) = d_x \log E(p_+, x) - d_x \log E(p_-, x)$
to rewrite the first term in \eqref{omega0-u} as
$$ - \sum_{r=2}^N  k_r  x_r \omega_{p_r - p_1}(u_n) =
- \sum_{r=1}^N k_r x_r d_x \log E(p_r, x) \mid_{x=u_n} 
+ (\lambda d) \hspace{2pt} d_x \log E(p_1, x)\mid_{x = u_n}.$$
Note also that when computing $\{\nu_{\lambda,n}, F \}$ 
using \eqref{conjugate-lambda} and \eqref{omega0-u}, the only 
contribution comes from the Poisson brackets
\begin{equation*}
\frac{\partial F}{\partial k_r}=\big\{\,x_r\,, \,F\, \big\},
\qquad - \frac{\partial F}{\partial \lambda_i}=\big\{\,\kappa_i\,,\, F\,\big\}.
\end{equation*}
All $\lambda$-dependent terms in $\nu_{\lambda, n}$ hence do not 
contribute to $\{\nu_{\lambda,n}, F \}$.
The first relation in \eqref{eqn-lem-cruc} now follows from direct comparison.
For the second relation, note that 
$\frac{\partial k_r}{\partial u_0} = \frac{k_r}{u_0}$, 
and hence
$$ \left\{ \lambda d, F \right\} 
= \left\{ \sum_{r=1}^N x_r k_r, F \right\}
= \sum_{r=1}^N k_r \frac{\partial F}{\partial k_r} 
= \sum_{r=1}^N u_0 \frac{\partial k_r}{\partial u_0} 
\frac{\partial F}{\partial k_r} 
= u_0 \frac{\partial F}{\partial u_0}. $$
\end{proof}

\begin{proof}(Proof of the main theorem)
The relations
\begin{align*}
    &\{u_n, \nu_{\lambda, m} \} = \delta_{nm},
    &\{u_0, \nu_{\lambda, 0} \} = u_0
\end{align*}
follow from Lemma \ref{lemma-crucial}.
The relations $\{u_n, u_m \} = 0$ follow from observation that $u_n = u_n(\blambda, \bmk)$ 
while $\lambda_1, \dots, \lambda_g, k_1, \dots, k_N$ are Poisson commuting.
The relations $\{\nu_{\lambda, n}, \nu_{\lambda, m} \} = 0$ 
follow from an argument using grading 
on the algebra of 
polynomial functions in variables $k_r$ and $\kappa_i$ (cf. \cite{DT1}).
This shows that $\widetilde{\SOVdR}$ and consequently $\SOVdR$ are Poisson maps.
That these maps are dominant follows from Proposition \ref{prop-inverse-SOVdR}.
\end{proof}

\section{Discussion}
\subsection{Relation to wobbly bundles} \label{sect-wobbly}
In constructing an analytic space that parametrises opers with 
$n + 3g-3$ apparent 
singularities on Riemann surfaces with $n$ punctures, Iwasaki identified 
and removed the singular loci. 
This locus is defined by effective divisors of degree $n + 3g-3$
such that there exist quadratic differentials vanishing there 
and having poles at the specified punctures
(cf. equation (6.3) in \cite{Iwa91}).
In the compact Riemann surfaces setting, the analogous locus is formed by 
\begin{equation}\label{wobbly-locus-2}
    \mathcal{M}_{\mathrm{op}}^{3g-3}(Q) \coloneqq
    \{ [(u_1, \nu_1), \dots, (u_{3g-3}, \nu_{3g-3})] \mid \sum_{n=1}^{3g-3} u_n 
\text{ is $Q$-special} \} \subset \mathcal{M}_{\mathrm{op}}^{3g-3}.
\end{equation}
This is a regular locus in the affine bundle $\Mop$ 
defined explicitly via \eqref{residue-transform} but needs to be 
removed if one wants a well-behaved parameter space for opers 
with $3g-3$ apparent singularities.
A quick argument to convince oneself is to observe that the nonhomogeneous 
linear system \eqref{linear-system} becomes degenerate at this locus: 
not all vectors $\Vec{\nu}$ give opers with corresponding residue parameters, 
and in case there is an oper with such residue parameters it is not unique.

The locus $\mathcal{M}_{\mathrm{op}}^{3g-3}(Q)$
defined in \eqref{wobbly-locus-2} is related to one component of 
the so-called wobbly divisor on $\Nstable$. A bundle $E$ is called 
very stable if it does not admit nonzero nilpotent Higgs fields, 
and wobbly if it is stable but not very stable \cite{DP09}.
The significance of wobbly bundles has been noted early on:
Laumon showed that very stable bundles are stable, 
and Drinfeld conjectured that the wobbly locus $\mathcal{W}$ is a divisor on 
the moduli space of stable vector bundles \cite{Laumon88}.
Pal-Pauly \cite{Pal-Pauly} proved this conjecture for $G=SL_2(\bC)$,
namely for the case of stable rank-2 bundles having fixed determinant $\Lambda$ 
which can be assumed to be of degree $0$ or $1$.
They furthermore showed that 
\begin{equation}\label{wobbly-components}
	\mathcal{W} = 
	\begin{cases}
		\mathcal{W}_{\deg(\Lambda)} \cup \mathcal{W}_{\deg(\Lambda) + 2} \cup \dots \cup \mathcal{W}_{g} 
		& \text{for } g\equiv \deg(\Lambda) \text{ mod } 2,  \\
		\mathcal{W}_{\deg(\Lambda)} \cup \mathcal{W}_{\deg(\Lambda) + 2} \cup \dots \cup \mathcal{W}_{g - 1} 
		& \text{for } g\equiv \deg(\Lambda) -1 \text{ mod } 2,
	\end{cases}
\end{equation}
where $\mathcal{W}_k$ is the closure in $\mathcal{N}_\Lambda$ of the locus 
of bundles admitting some subbundle $L$ such that $\deg(K L^2 \Lambda^{-1}) = k$. 
All $\mathcal{W}_k$ are irreducible divisors except for 
$\mathcal{W}_0$ which is an union of $2^{2g}$ irreducible divisors. 
The following result from \cite{D} gives a criterion for producing wobbly 
bundles by taking direct images of distinguished line bundles 
from a spectral curve.

\begin{theorem}\cite{D}
Consider a smooth spectral curve $S \overset{\pi}{\rightarrow} X$ defined by a 
quadratic differential on $X$, and $\widetilde{\bu}$ an effective 
divisor on $X$ such that $\bu = \pi(\widetilde{\bu})$ is $Q$-special 
and $2g-2 < \deg(\bu) \leq 4g-4$.
Then the rank-2 bundle $\pi_\ast(\mathcal{O}_S(\widetilde{\bu}))$ on $X$ 
is wobbly, 
provided that it is stable (which is the generic case). 
In this case, upon a twist by a line bundle to adjust the determinant, 
$\pi_\ast(\mathcal{O}_S(\widetilde{\bu}))$ is contained in 
$\mathcal{W}_{4g-4-\deg(\bu)}$.
\end{theorem}

For $S \subset \Mop$
denote by $i \circ \SOVdR^{-1}(S)$ the locus in $\Nstable$ 
defined by projecting $\SOVdR^{-1}(S)$ to $\N$ and then to $\Nstable$.
We have the following.

\begin{proposition}
\begin{enumerate}
    \item If $(E, L, \nabla_\lambda)$ produces an oper with 
$3g-3$ simple apparent singularities that are zeroes of a quadratic differential,
then $E$ is wobbly with $L$ the kernel of some nilpotent Higgs fields on $E$.
    \item The closure in $\Nstable$ of the locus
$i \circ \SOVdR^{-1}(\mathcal{M}_{\mathrm{op}}^{3g-3}(Q))$
is the wobbly component $\mathcal{W}_{g-1}$.
\end{enumerate}
\end{proposition}

\subsection{Families of $\lambda$-connections}
\label{sect-families-connections}
By construction, fixing a reduced effective divisor $\bu$ on $X$ 
defines a Lagrangian (Poisson) leaf in $(T^\ast X)^{[3g-3]}$ 
(respectively, $\mM_{\mathrm{op}}^{3g-3}$).
In this section, we will attempt to have a rather preliminary discussion on 
the space of Higgs bundles $(E,\phi) \in \mM_H$ 
($\lambda$-connections $(E, \nabla_\lambda) \in \mM_\lambda$)
that induce these leaves.
We will sketch a relation between these two spaces 
and draw comparison to the Lagrangian leaves defined by the 
$\bC^\ast$ on the Hodge moduli space $\mM_{Hod}$
of $\lambda$-connections.

\subsubsection{Lagrangian leaves and the conformal limit}
\paragraph{Lagrangian upward flows in Hitchin moduli space}
Recall from \cite{Hit} that 
the Hitchin moduli space $\mM_H$ admits a 
$\bC^\ast$-action that scales the Higgs fields. 
A Higgs bundle $\cE = (E_0, \phi_0) \in \mM_H^{\bC^\ast}$ has an upward flow 
$W^0_{\cE}$
consisting of Higgs bundles $(E,\phi)$ with 
$\underset{\epsilon \rightarrow 0}{\lim} [(E, \epsilon \phi)] = [\cE]$.
For example, for $\cE = (E_0, 0)$ with $E_0$ stable,
we have $W^0_\cE$ consist of Higgs fields on $E_0$,
and for $\cE = (E_0, \phi_0)$ with $E_0$ destabilised by a subbundle 
$L_0$, we have $W^0_\cE$ consist of Higgs bundles $(E,\phi)$
where $E$ is also destabilised by $L_0$ and the zero divisor of $c(E, L_0, \phi)$ 
coincides with that of $c(E_0, L_0, \phi_0)$. 

On the other hand, the $\bC^\ast$-action acts linearly on $T_\cE \mM_H$.
Consider the weight decomposition 
$T_\cE \mM_H = \oplus_{k \in \mathbb{Z}} {{(T_\cE} \mM_H)}_k$ 
where ${{(T_\cE} \mM_H)}_k$ is the weight space in which $\epsilon \in \bC^\ast$ 
acts as $\epsilon^k$, and let 
$V_\cE \coloneqq \oplus_{k >0} {{(T_\cE} \mM_H)}_k$
be the positive weight subspace.
\begin{proposition} \label{prop-weight-space} \cite{BB, HH22}
If $\cE$ is stable then $W^0_\cE \simeq V_\cE \simeq \bC^{3g-3}$ 
as varieties with $\bC^\ast$-action.
Furthermore, $W^0_\cE$ is a Lagrangian in $\mM_H$.
\end{proposition}

Note in particular that the $\bC^\ast$-fixed point $\cE$ is naturally identified 
with the origin in $V_\cE \simeq \bC^{3g-3}$.

\paragraph{Lagrangians in moduli space $\mM_\lambda$ of $\lambda$-connections}
Simpson studied the $\bC^\ast$-action on 
the Hodge moduli space $\MHodge$ of $\lambda$-connections 
with varying $\lambda \in \bC$ 
and similarly define the upward flows $W_\cE$
to $\bC^\ast$-fixed points $\cE = (E_0, \phi_0)$ in
$\mM_H \equiv \mM_{\lambda = 0}$. 
For $E_0$ stable and $\phi_0 = 0$,
$W_\cE$ consists of all irreducible $\lambda$-connections on $E_0$,
and for $E_0$ destabilised by subbundle $L$,
$W_\cE$ consists of $(E, \nabla_\lambda)$ 
where $E$ is also destabilised by $L$ 
and the zero divisors of $c(L, E, \nabla_\lambda)$ 
and $c(L, E_0, \phi_0)$ coincide.
The restriction of $W_\cE$ to $\mM_H \equiv \mM_{\lambda = 0}$ 
defines the upward flows $W^0_\cE$;
we similarly denote by $W^\lambda_\cE$ the restriction of $W_\cE$ 
to the moduli space $\mM_\lambda$ 
of $\lambda$-connections with fixed $\lambda \in \bC^\ast$.
For $\lambda \neq 0$, Simpson conjectured that
the leaves $W^\lambda_\cE$ are closed \footnote{The construction of Lagrangian leaves hold for all ranks.
Simpson's conjecture on the closed-ness of $W^\lambda_\cE$ 
for $\lambda \neq 0$ was proved for the rank-2 cases in \cite{DS}.
}
in $\mM_\lambda \equiv \mM_{dR}$.

\paragraph{Conformal limit and biholomorphism between Lagrangian leaves}
Let us recall from \cite{Cor88, S88} that the nonabelian Hodge correspondence 
is a diffeomorphism between the Hitchin and de Rham moduli spaces
$$ \mathrm{NAH}: \mM_H \equiv \mM_0 \longrightarrow \mM_{dR} \equiv \mM_1 $$
obtained by solving Hitchin's self-duality equations 
which define stable Higgs bundles.
It is known that for a $\bC^\ast$-fixed stable Higgs bundle $\cE$,
$$ \mathrm{NAH}(\cE) \in W^1_\cE. $$
In \cite{CW}, Collier-Wentworth showed that
there is a biholomorphic embedding 
\begin{equation}\label{pHod}
    p_{Hod}: V_\cE \times \bC \overset{\sim}{\longrightarrow} W_\cE.     
\end{equation}
The restrictions $p^0_{Hod}$ and $p^\lambda_{Hod}$ 
of $p_{Hod}$ to $V_\cE \times \{0\}$ and $V_\cE \times \{\lambda \}$ 
for some $\lambda \in \bC^\ast$ 
\begin{equation}
\begin{tikzcd}
    & 
    &V_\cE \arrow{dl}[swap]{p^0_{Hod}} \arrow{dr}{p^\lambda_{Hod}} 
    & \\
    &W^0_\cE 
    & 
    &W^\lambda_\cE 
\end{tikzcd}
\end{equation}
are biholomorphisms satisfying
\begin{itemize}
\item $p^0_{Hod}$ is the isomorphism $V_\cE \simeq W_\cE$ 
in Proposition \ref{prop-weight-space};
\item $p^\lambda_{Hod} \circ (p^0_{Hod})^{-1} (\cE)$ corresponds to 
$\mathrm{NAH}(\cE)$ scaled by $\lambda$.
\end{itemize}
We then see that the map
$$\mathrm{CL}_\lambda \coloneqq p^\lambda_{Hod} \circ (p^0_{Hod})^{-1}
: W^0_\cE \longrightarrow W^\lambda_\cE$$
is a biholomorphism that can be constructed in two steps:
\begin{enumerate}
\item Use (the scaling of) $\mathrm{NAH}(\cE)$ as a ``reference connection'', 
i.e. a counter-part of the ``origin'' $\cE$; 
\item Use $V_\cE$ as a deformation space on both $\mM_H$ and $\mM_\lambda$ 
to cover $W^0_\cE$ and $W^\lambda_\cE$ respectively.
\end{enumerate}

We have used the notation $\mathrm{CL}_\lambda$ as this map is related to 
the $\hbar$-conformal limit 
defined by Gaiotto \cite{Gai, DFK+}, 
which is the limit of a family of flat connections,
i.e. elements of $\mM_{dR} \equiv \mM_1$, 
induced by the nonabelian Hodge correspondence 
\footnote{More precisely, the family of flat connections is induced by
applying two scaling to the nonabelian Hodge correspondence:
one by the $\bC^\ast$-action on $\mM_H$ and one by 
the so-called real twistor line in $\MHodge$.}
of a Higgs bundle $(E, \phi)$.
Collier-Wentworth showed that for any Higgs bundle $(E,\phi)$ 
with stable $\bC^\ast$-limit,
its $\hbar$-conformal limit is equal to 
$\hbar^{-1} . \mathrm{CL}_{\hbar}((E, \phi))$.
Note that in the formulation of the conformal limit,
the technical step is to show  
existence of the limit of the families of flat connections 
and then compute this limit \cite{CW, DFK+}.
On the other hand, in the work of Collier-Wentworth 
\cite{CW}, the technical part was to show existence of a deformation space 
$V_\cE$ that can be embedded in both $\mM_H$ and $\mM_{dR}$ to carry out Step 2 above.

\paragraph{Examples of conformal limit}
\begin{example}\label{ex-CL-stable-bundle}
Let $E_\lambda = E$ a fixed stable bundle for all $\lambda \in \bC$,
and $\nabla^{\mathrm{ref}}_1$ the holomorphic connection 
defined by the nonabelian Hodge correspondence 
of the Higgs bundle $(E,0)$.
For each $\lambda \in \bC$,
define $\nabla^{\mathrm{ref}}_\lambda = \lambda. \nabla^{\mathrm{ref}}_1$;
in particular $\nabla^{\mathrm{ref}}_0$ is the Higgs bundle $\cE = (E,0)$.
The deformation space is $V_\cE = H^0(X, \mathrm{End}_0(E)) \simeq \bC^{3g-3}$, 
and so the deformation $\delta \in V_\cE$ are trace-free Higgs fields on $E$.
The assignment
\begin{align*}
\mathrm{CL}_\lambda 
: W^0_\cE &\longrightarrow W^\lambda_\cE \nonumber \\
\left( E, \delta \right) &\longmapsto  
\left(E, \nabla^{\mathrm{ref}}_\lambda + \delta  \right) 
\end{align*}
then defines a biholomorphic map between the space of Higgs fields on $E$ 
and the space of irreducible $\lambda$-connections on $E$. 
\end{example}

\begin{example}\label{ex-CL-Hitchin-section}
For $\lambda \in \bC^\ast$, 
let $E'$ 
be the bundle realised as the non-trivial extension of $K^{-1/2}_X$ 
by $K^{1/2}_X$, which is unique up to scaling, 
and let $E_0 = K^{1/2}_X \oplus K^{-1/2}_X$.
Let $\cE = (E_0, \phi_0)$ 
with $\phi_0 = \left( \begin{smallmatrix} 0 & 0 \\ 1 & 0 
\end{smallmatrix} \right)$, 
and denote by $(E', \nabla\connref_\lambda)$ 
the $\lambda$-scaling of the nonabelian Hodge correspondence 
$(E', \nabla\connref_1)$ of $\cE$.
Note that $(E', \nabla\connref_1)$ 
is often called the ``uniformising oper'', which in coordinates on $X$ 
induced by coordinates on its universal covering -- the upper-half plane -- 
takes the form $\nabla\connref_1 = \partial + 
\left( \begin{smallmatrix} 0 & 0 \\ 1 & 0 \end{smallmatrix} \right)$ 
in local frames adapted to $K^{1/2}_X$.
Let $V_\cE$ be the space of quadratic differentials on $X$ and,
for $q \in V_\cE$, abusing the notation denote by 
$\delta_q = \left( \begin{smallmatrix}
0 & q \\ 0 & 0 \end{smallmatrix} \right)$ 
both the corresponding Higgs fields on $E$ and $E'$.
The assignment
\begin{align*}
\mathrm{CL}_\lambda 
: W^0_\cE &\longrightarrow W^\lambda_\cE \nonumber \\
\left( E_0, \phi_0 + \delta_q \right) 
&\longmapsto \left( E', \nabla\connref_\lambda + \delta_q \right)
\end{align*}
defines a biholomorphic map 
between the Hitchin section and the space of $\lambda$-opers
\begin{align}
    W^0_\cE = 
    \left\{ \left( E_0, \begin{pmatrix} 0 & q \\ 1 & 0 \end{pmatrix} \right) \right\}
    \longrightarrow 
    W^\lambda_\cE = 
    \left\{ \left( E', 
    \lambda\partial + \begin{pmatrix} 0 & q \\ 1 & 0 \end{pmatrix} \right) \right\}
\end{align}
(upon choosing a square-root $K^{1/2}_X$).
\end{example}

\subsubsection{Families of $\lambda$-connections with fixed zero divisor $c(E, L, \nabla_\lambda)$}
Examples \ref{ex-CL-Hitchin-section} and \ref{ex-CL-stable-bundle} demonstrate 
the fact that for a family $\{(E_\lambda, \nabla_\lambda) \}_{\lambda \in \bC}$ of 
$\lambda$-connections, if one wants to keep fixed one of the data
\begin{itemize}
    \item the zero divisor of $c \equiv c(L_\lambda \hookrightarrow E_\lambda, \nabla_\lambda)$,
    \item the isomorphism class of $E_\lambda$
\end{itemize}
for all $\lambda \in \bC$,
then the other data has a ``jump'' as $\lambda \rightarrow 0$  
(cf. \eqref{Serre-constraint}).
In particular, if the family is defined by $\bC^\ast$-action, 
then $c\neq 0$ and $\divisor(c)$ is fixed 
only if $E_\lambda$ are destabilised by $L$.

Let us observe that, however, 
we could define a family $\{(E_\lambda, \nabla_\lambda) \}_{\lambda \in \bC}$ where 
$\divisor(c)$ is fixed 
and $E_\lambda$ are generically stable for all $\lambda \in \bC$ 
provided that we are willing to let $E_\lambda$ vary. 
The following is a simple example.

\begin{example}\label{ex-family-lambda-conn}
Let $E \equiv E_0$ be a stable bundle,  
$L$ a maximal subbundle of $E_0$, and $\phi$ a Higgs field on $E$. 
Denote by $\bmx_0 \equiv (L \hookrightarrow E) \in H^1(X, L^2)$ 
an extension class realising $E$ and $c \equiv c(\bmx_0, \phi)$.
Pick an extension class $\bmx_1\in H^1(X, L^2)$ 
which satisfies $\langle \bmx_1, c \rangle = 1$
(so in particular it not contained 
in the hyperplane $\ker(c) \subset H^1(X, L^2) $).
For $\lambda \in \bC$,
consider a family of bundles $E_\lambda$ realised by extension classes 
$\bmx_\lambda$ where
\begin{align}
&\bmx_\lambda \coloneqq \bmx_0 
+ \lambda \deg(L) \bmx_1,
&\lambda \in \bC.
\end{align}
A generic element $E_\lambda$ in this family is stable 
and admit $L$ as a maximal subbundle \cite{Lange-Narasimhan};
in this case, $E_\lambda$ is in particular not ``overcounted'' in $H^1(X, L^2)$, 
namely $h^0(X, L^{-1} E_\lambda) = 1$.
Suppose that the condition $h^0(X, L^{-1} E_\lambda) = 1$
holds for all $\lambda \in \bC$. 
It follows from Proposition \ref{prop-surjection-c-2} that
we can define a family $\{(E_\lambda, \nabla_\lambda) \}_{\lambda \in \bC}$
where $c(E_\lambda, L, \nabla_\lambda) = c$.    
\end{example}

Note that this family is not unique as the choice of $\bmx_1$ 
is not unique -- one can in fact form a basis for $H^1(X, L^2)$
consisting of such elements -- 
and even upon fixing such $\bmx_1$
there exist different choices of $\nabla_\lambda$ whose 
difference is an $L$-invariant Higgs field -- 
these are of dimension $h^0(X, E^\ast LK)$ 
(cf. Propositions \ref{prop-surjection-c-Higgs} and \ref{prop-surjection-c-2})

\paragraph{Deformation spaces in the Hitchin moduli space}
Consider a reduced divisor $\bu = \sum_{n=1}^{3g-3} u_n$ on $X$, 
i.e. all $u_n$ have multiplicity $1$.
The deformation space associated to $\bu$ is 
the space $\mM_H(\bu) \subset \mM_H$ defined by 
$$\mM_H(\bmu) \coloneqq 
\{ (E,\phi) \in \mM_H \mid \exists L \hookrightarrow E, \hspace{2pt}
\mathrm{div}(c(E, L, \phi)) = \bmu \}. $$
The subbundle $L$ in the definition of $\mM_H(\bu)$ is a   
square-root of $K_X \otimes \mathcal{O}_X(- \bu)$.
For each $(E, \phi) \in \mM_H(\bu)$ there exists a unique 
subbundle $L$ of $E$ such that 
$c(E, L, \phi)$ has zero divisor $\bu$,
and so $\mM_H(\bu)$ has $2^{2g}$ connected components 
each corresponding to a choice of the square-root.
Consider the dense subset $\mM^s_H(\bu) \subset \mM_H(\bu)$ 
defined by smooth spectral curves. 
The composition of $\SOV$ 
with the pull-back of $(E, \phi) \in \mM_H(\bu)$ to $T^\ast_{(E,L)}\N$
defines a map
$$ \SOV(\bu): \mM^s_H(\bu) \longrightarrow \bC_\bu $$
where 
$$\bC_\bu \coloneqq \prod_{i=1}^{3g-3} T^\ast_{u_i} X  
\simeq \bC^{3g-3}$$
is a Lagrangian in $(T^\ast X)^{[3g-3]}$. 
Denote by $\mM^{ss}_H(\bu)$ the open dense subset in $\mM^s_H(\bu)$ 
defined by Higgs bundles $(E, \phi)$ where $E$ is stable. 
\begin{proposition}
Let $\bu = \sum_{n=1}^{3g-3} u_n$ be a $Q$-generic divisor. Then
\begin{enumerate}
    \item The map $\SOV(\bu)$ is a $2^{2g}:1$ unbranched holomorphic covering 
    onto its image which is dense in $\bC_\bu$. 
    \item $\mM^{ss}_H(\bu)$ is a Lagrangian in $\mM_H$.
\end{enumerate}
\end{proposition}
\begin{proof}
\begin{enumerate}
\item 
Observe that any point in $\bC_\bu$ defines a spectral curve,
and in particular an open dense subset of $\bC_\bu$ defines 
smooth spectral curves.
The claim now comes from the invertibility of the construction of Baker-
Akhiezer divisors (cf. Theorem \ref{thm-BA})) up to tensoring with 
square-roots of $\mathcal{O}_X$. 
\item The symplectic structure of $\mM_H$ at the loci defined by $(E, \phi)$ 
where $E$ is stable coincides with the canonical symplectic structure 
of $T^\ast \Nstable$. 
The claim now follows from the Theorem \ref{thm-SOV-Higgs} 
and the fact that $i: \N \dashrightarrow \Nstable$ is a local isomorphism 
away from the loci that makes $\bu$ a $Q$-special divisor 
(cf. Proposition \ref{prop-immersion}).
\end{enumerate}
\end{proof}

\begin{remark}
\begin{enumerate}
\item Given a Higgs bundle $(E,\phi) \in \mM^{ss}_H(\bu)$, 
its scaling $\epsilon. (E, \phi)$ for some $\epsilon\in \bC^\ast$
is also contained in $\mM^{ss}_H(\bu)$.
The closure of $\mM^{ss}_H(\bu)$ in $\mM_H$ therefore 
contains the $\bC^\ast$-fixed point $(E,0)$ in the nilpotent cone.
\item It is natural to expect that the Lagrangian property of 
$\mM^{ss}_H(\bu)$ extends to all of $\mM_H(\bu)$.
\end{enumerate}
\end{remark}

\paragraph{Analogues of deformation spaces in the moduli space of $\lambda$-connections?}
For a fixed $\lambda \neq 0$, 
in the moduli space $\mM_\lambda$ of $\lambda$-connections 
one can consider an analogous subspace
\begin{equation*}
    \mM_{\lambda}(\bmu) \coloneqq 
    \{ (E,\nabla_\lambda) \in \mM_{\lambda} 
    \mid \exists L \hookrightarrow E, \hspace{2pt}
    \mathrm{div}(c(E, L, \nabla_\lambda)) = \bmu \},
\end{equation*}   
which also has $2^{2g}$ connected components 
corresponding to square-roots of $K_X \otimes \mathcal{O}_X(- \bu)$.
The map $\SOVdR$ then allows us to define 
the analogue of $\SOV(\bu)$, namely 
$$ \SOVdR(\bu): \mM_{\lambda}(\bu) \longrightarrow 
\bC_{\lambda, \bu} \coloneqq 
\prod_{i=1}^{3g-3} {\mM'_{\mathrm{op}}}\mid_{u_i} $$
where $\mM'_{\mathrm{op}}$ is the affine bundle on $X$ modeled over $T^\ast X$
defined via \eqref{residue-transform-scaled}. By construction, 
the target space $\bC_{\lambda, \bu}$ 
is an affine space modeled over the vector space 
$\bC_\bu \simeq \bC^{3g-3}$ and defines a Lagrangian in 
$\mM_{\mathrm{op}}^{3g-3}$. 

Suppose now $\bu$ is $Q$-generic.
Note that by Proposition \ref{prop-characterise-opers}, 
a point in $\bC_{\lambda, \bu}$ corresponds 
uniquely to an oper with apparent singularities at $\bu$. 
By Proposition \ref{prop-inverse-SOVdR},
one can find a $\lambda$-connection $(E, \nabla_\lambda)$ 
that induces such an oper, and if $(E, \nabla_\lambda) \in \mM_\lambda$
(i.e. it is irreducible)
then there are exactly $2^{2g}$ such $\lambda$-connections, 
which differ from each other by a twist by a square-root of $\mathcal{O}_X$.
Denote by $\bC^s_\bu \subset \bC_\bu$ the image of $\mM^s_H(\bu)$ along $\SOV(\bu)$
(note that $0 \notin \bC^s_\bu)$.
We expect that there exists a dense subspace $Y \subset \bC_{\lambda, \bu}$
such that 
\begin{itemize}
\item $Y$ is biholomorphic to $\bC^s_\bu$;
\item the preimage $\mM_\lambda'(\bu) \coloneqq (\SOVdR(\bu))^{-1}(Y)$ 
is a Lagrangian in $\mM_\lambda$ and is biholomorphic to $\mM_H^{s}(\bu)$.
\end{itemize}
\vspace{-1mm}
If this is true, then we can regard $\bC^s_\bu$ as a deformation space 
which is holomorphically embedded in both $\mM_H$ and $\mM_\lambda$.
A (non-canonical) choice of family of $\lambda$-connections as in Example 
\ref{ex-family-lambda-conn} 
would then define an analogue of the embedding $p_{Hod}$ defined in \eqref{pHod}.

\vspace{4mm} \noindent
\textit{Acknowledgment:}
I would like to thank Jörg Teschner for helpful discussions 
and feedback 
-- this work comes out from our joint research program -- 
and Vladimir Roubtsov for discussions on related subjects.
Part of this work is extracted from my PhD research  
which was funded by the DFG Cluster of Excelence ``Quantum Universe'' 
at the University of Hamburg.
I also thank the  Oberwolfach Research Institute for Mathematics 
and Max Planck Institute for Mathematics in Bonn
for funding.

\noindent
\textit{Contact:} \textit{duongdinh.mp@gmail.com}

\end{document}